\documentclass[12pt,oneside,reqno]{amsart}

\usepackage{amsmath,amssymb,amsthm,textcomp}
\usepackage{amsfonts,graphicx}
\usepackage[mathscr]{eucal}
\usepackage{color}
\usepackage{csquotes}
\usepackage[backend=bibtex,%
firstinits=true,%
doi=false,%
isbn=true,%
url=false,%
maxnames=99]{biblatex}%

\AtEveryBibitem{\clearfield{issn}}
\AtEveryCitekey{\clearfield{issn}}
\numberwithin{equation}{section}
\DeclareNameAlias{sortname}{last-first}
\theoremstyle{definition}
\usepackage{mathtools}
\numberwithin{equation}{section}

\newcommand{\ncom}{\newcommand}

\ncom{\beq}{\begin{equation}}
\ncom{\eeq}{\end{equation}}
\ncom{\bea}{\begin{eqnarray*}}
\ncom{\eea}{\end{eqnarray*}}
\ncom{\beqa}{\begin{eqnarray}}
\ncom{\eeqa}{\end{eqnarray}}
\ncom{\nno}{\nonumber}
\ncom{\non}{\nonumber}
\ncom{\ds}{\displaystyle}
\ncom{\half}{\frac{1}{2}}
\ncom{\mbx}{\makebox{.25cm}}
\ncom{\hs}{\mbox{\hspace{.25cm}}}
\ncom{\rar}{\rightarrow}
\ncom{\Rar}{\Rightarrow}
\ncom{\noin}{\noindent}
\ncom{\bc}{\begin{center}}
\ncom{\ec}{\end{center}}
\ncom{\sz}{\scriptsize}
\ncom{\rf}{\ref}
\ncom{\s}{\sqrt{2}}
\ncom{\sgm}{\sigma}
\ncom{\Sgm}{\Sigma}
\ncom{\psgm}{\sigma^{\prime}}
\ncom{\dt}{\delta}
\ncom{\Dt}{\Delta}
\ncom{\lmd}{\lambda}
\ncom{\Lmd}{\Lambda}
\ncom{\Th}{\Theta}
\ncom{\e}{\eta}
\ncom{\eps}{\epsilon}
\ncom{\pcc}{\stackrel{P}{>}}
\ncom{\lp}{\stackrel{L_{p}}{>}}
\ncom{\dist}{{\rm\,dist}}
\ncom{\sspan}{{\rm\,span}}
\ncom{\re}{{\rm Re\,}}
\ncom{\im}{{\rm Im\,}}
\ncom{\sgn}{{\rm sgn\,}}
\ncom{\ba}{\begin{array}}
\ncom{\ea}{\end{array}}
\ncom{\hone}{\mbox{\hspace{1em}}}
\ncom{\htwo}{\mbox{\hspace{2em}}}
\ncom{\hthree}{\mbox{\hspace{3em}}}
\ncom{\hfour}{\mbox{\hspace{4em}}}
\ncom{\vone}{\vskip 2ex}
\ncom{\vtwo}{\vskip 4ex}
\ncom{\vonee}{\vskip 1.5ex}
\ncom{\vthree}{\vskip 6ex}
\ncom{\vfour}{\vspace*{8ex}}
\ncom{\norm}{\|\;\;\|}
\ncom{\integ}[4]{\int_{#1}^{#2}\,{#3}\,d{#4}}
\ncom{\vspan}[1]{{{\rm\,span}\{ #1 \}}}
\ncom{\dm}[1]{ {\displaystyle{#1} } }
\ncom{\ri}[1]{{#1} \index{#1}}

\newtheorem{theorem}{\bf Theorem}[section]
\newtheorem{remark}{\bf Remark}[section]
\newtheorem{proposition}{Proposition}[section]
\newtheorem{lemma}{Lemma}[section]
\newtheorem{corollary}{Corollary}[section]
\newtheorem{example}{Example}[section]
\newtheorem{definition}{Definition}[section]
\newtheoremstyle
    {remarkstyle}
    {}
    {11pt}
    {}
    {}
    {\bfseries}
    {:}
    {     }
    {\thmname{#1} \thmnumber{#2} }

\theoremstyle{remarkstyle}

\def\l{{\langle}}
\def\r{\rangle}

\def\eps{\varepsilon}

\begin{document}
\title{Iterated Generalized Counting Process and its Extensions}
\author[Manisha Dhillon]{Manisha Dhillon}
\address{Manisha Dhillon, Department of Mathematics, Indian Institute of Technology Bhilai, Durg 491002, India.}
\email{manishadh@iitbhilai.ac.in}
\author[Kuldeep Kumar Kataria]{Kuldeep Kumar Kataria}
\address{Kuldeep Kumar Kataria, Department of Mathematics, Indian Institute of Technology Bhilai, Durg 491002, India.}
\email{kuldeepk@iitbhilai.ac.in}
\date{\today}
\begin{abstract}
In this paper, we study the composition of two independent GCPs which we call the iterated generalized counting process (IGCP). Its distributional properties such as the transition probabilities, probability generating function, state probabilities and its corresponding L\'evy measure are obtained. We study some integrals of the IGCP. Also, we study some of its extensions, for example, the compound IGCP, the multivariate IGCP and the $q$-iterated GCP. It is shown that the IGCP and the compound IGCP are identically distributed to a compound GCP which leads to their martingale characterizations. Later, a time-changed version of the IGCP is considered where the time is changed by an inverse stable subordinator. Using its covariance structure, we establish that the time-changed IGCP exhibits long-range dependence property. Moreover, we show that its increment process exhibits short-range dependence property. Also, it is shown that its one-dimensional distributions are not infinitely divisible. Initially, some of its potential real life applications are discussed.
\end{abstract}

\maketitle
\section{Introduction}
In the past two decades, the time-changed processes have attracted the interest of several researchers because of their potential applications across various fields such as finance, biology, hydrology, internet data traffic modeling, \textit{etc}. The  Poisson process time-changed by an independent stable subordinator (see Orsingher and Polito (2012a)) and by its first hitting time (see Beghin and Orsingher (2009), Laskin (2003)) are two extensively studied time-changed  processes. These are known as the space fractional Poisson process and the time fractional Poisson process, respectively. 

Di Crescenzo {\it et al.} (2016) introduced a generalization of the time fractional Poisson process which performs independently $k$ kinds of jumps of amplitude $1,2,\dots,k$ with positive rates $\lambda_1,\lambda_2,\dots,\lambda_k$, respectively. It is known as the generalized fractional counting process (GFCP) and we denote it by $\{M^\alpha(t)\}_{t\geq0}$, $0<\alpha\le1$. Its state probabilities $p^\alpha(n,t)=\mathrm{Pr}\{M^\alpha(t)=n\}$ satisfy the following system of differential equations:
\begin{equation}\label{gfcpdeq}
\frac{\mathrm{d}^{\alpha}}{\mathrm{d}t^{\alpha}}p^{\alpha}(n,t)=-\sum_{j=1}^{k}\lambda_jp^{\alpha}(n,t)+	\sum_{j=1}^{\min\{n,k\}}\lambda_{j}p^{\alpha}(n-j,t),\ \ n\ge0
\end{equation}
with $p^{\alpha}(n,0)=\delta_n(0)$, where $\delta_n$'s are Dirac measures.

The fractional derivative involved in \eqref{gfcpdeq} is the Caputo fractional derivative defined in \eqref{caputo}. Its involvement induces a global memory effect in the system.

The state probabilities of GFCP are given by (see Di Crescenzo \textit{et al.} (2016))
\begin{equation}\label{GFCPPMF}
p^\alpha(n,t)=\sum_{r=0}^{n}\sum_{\substack{x_1+x_2+\dots+x_k=r\\
x_1+2x_2+\dots+kx_k=n}}r!\Big(\prod_{j=1}^{k}\frac{(\lambda_jt^\alpha)^{x_j}}{x_j!}\Big)E_{\alpha,r\alpha+1}^{r+1}(-\lambda t^\alpha),\ n\ge0,
\end{equation}
where $x_j$'s are non-negative integers and $E_{\alpha,r\alpha+1}^{r+1}(\cdot)$ is the three-parameter Mittag-Leffler function defined in \eqref{mitag}. Also, its mean is given by 
\begin{equation}
	\mathbb{E}(M^\alpha (t))=\sum_{j=1}^{k}j\lambda_j\frac{t^\alpha}{\Gamma(\alpha+1)}\label{meangfcp}
	\end{equation}
and its variance is given by
\begin{equation}
 \operatorname{Var}(M^\alpha(t))=\Big(\sum_{j=1}^{k}j\lambda_jt^\alpha\Big)^2\left(\frac{2}{\Gamma(2\alpha+1)}-\frac{1}{\Gamma^2(\alpha+1)}\right)+\sum_{j=1}^{k}j^{2}\lambda_j\frac{t^\alpha}{\Gamma(\alpha+1)}.\label{vargfcp}
\end{equation}

 For $\alpha=1$, the GFCP reduces to the generalized counting process (GCP), denoted by $\{M(t)\}_{t\ge0}$. Its transition probabilities are given by 
\begin{equation*}
	\mathrm{Pr}\{M(h)=j\}=\begin{cases}
		1-\sum_{j=1}^{k}\lambda_j h+o(h), \ j=0,\\
		\lambda_j h+o(h),\ 1\leq j\leq k,\\
		o(h), \ j\geq k+1,
	\end{cases}
\end{equation*}	
where $o(h)/h\to0 $ as $h\to0$. 

 For $k=1$, the GFCP and the GCP reduces to the time fractional Poisson process (see Mainardi (2004), Beghin and Orsingher (2009), Meerschaert {\it et al.} (2011)) and the homogeneous Poisson process, respectively. Recently, Dhillon and Kataria (2024) give martingale characterizations for the GCP and its time-changed variants. For additional properties of the GCP and its application in risk theory, we refer
 the reader to Kataria and Khandakar (2022). 

 Bochner (1955) introduced the concept of composition of independent processes. Orsingher and Polito (2012b) studied the compositions of
two independent Poisson processes and Di Crescenzo \textit{et al.} (2015) studied a compound Poisson process time-changed by an independent Poisson subordinator. They studied the first-crossing time problem through various types of
boundaries of the iterated Poisson process. Recently, Beghin and Orsingher (2016) introduced and studied the iterated birth process where they  considered the linear birth processes,
linear death processes and sublinear death processes time-changed by Poisson subordinator. For some recent works on the composition of independent processes, we refer the reader to Buchak and Sakhno (2017), Meoli (2023), and references therein. 

\subsection{Potential applications}
Here, we discuss some potential real life applications of the iterated processes studied in this paper.\\
\noindent (i) In a stock market, a trader can buy multiple stocks simultaneously, and also can sell multiple stocks simultaneously. Consider a trader who sells stocks according to the GCP $\{M(t)\}_{t\ge0}$, where $M(t)$ denotes the number of stocks sold by time $t$. Also, the trader buys stocks according the GCP $\{M_{0}(t)\}_{t\ge0}$ which is independent of $\{M(t)\}_{t\ge0}$. Here, $M_0(t)$ denotes the number of stocks purchased by time $t$. Then, the iterated process $\{M(M_0(t))\}_{t\ge0}$ represents the selling of stocks based on the number of stocks purchased by the trader.\\
\noindent (ii)  In oncology, tumor cells often divide uncontrollably and multiple cells may undergo mitosis simultaneously. Suppose a patient diagnosed with cancer is under treatment, and let the growth of cancer cells inside the body of patient is modeled by the GCP $\{M_0(t)\}_{t\ge0}$. During the treatment, the tumor cells get killed according to the GCP $\{M(t)\}_{t\ge0}$ which is independent of $\{M_0(t)\}_{t\ge0}$. Let $Y_1$ denote the severity of regrowth of tumor cells. Then, the compound process $\{\sum_{i=1}^{M(M_0(t))}Y_i\}_{t\ge0}$ models the damage inflicted by the tumor on the patient where $Y_i$'s are iid random variables independent of $\{M(M_0(t))\}_{t\ge0}$.

In this paper, we introduce and study the iterated generalized counting process (IGCP), which is obtained by the composition of two independent GCPs. First, we give several distributional results related to the IGCP. We derive explicit expressions for its probability generating function (pgf) and obtain its state probabilities in terms of Bell polynomials. Also, we derive the system of differential equations that governs  the state probabilities of IGCP. Moreover, we obtain its associated L\'evy measure and calculate mean, variance, and the distribution of first-passage time. It is shown that the IGCP is identically distributed to a compound GCP which leads to its martingale characterization. Also, we study a fractional integral of the IGCP.  Further, a non-homogeneous version of the IGCP is considered where the time-changing component is a non-homogeneous GCP. 

Then, we study some extended versions of the IGCP that includes the compound version, the multivariate version and the $q$-iterated GCP. Some of their distributional properties such as pgf, pmf, mean, variance, \textit{etc.} are obtained. We establish that the compound IGCP is equal in distribution to a compound GCP which gives a martingale result for the compound IGCP. It is shown that the multivariate IGCP is a L\'evy process and its corresponding L\'evy measure is derived. Later, a time-changed variant of the IGCP is discussed where time-changing component is an inverse stable subordinator. We establish that the time-changed IGCP exhibits long-range dependence (LRD) property and its increment process exhibits short-range dependence (SRD) property. It is shown that its one-dimensional distributions are not infinitely divisible.

\section{Preliminaries}
In this section, we give some known results and definitions that will be used in this paper.
{Here, $\mathbb{R}$ and $\mathbb{N}_0$ denote the set of real numbers and non-negative integers, respectively.}
\begin{definition}
 Let $f$ and $g$ be two positive functions. The function $f(t)$ is said to be asymptotically equal to $g(t)$ if $\lim_{t\to \infty} f(t)/g(t)=1$. It is denoted by $f(t)\sim g(t)$ as $t\to \infty$.
\end{definition}
\subsection{Bell polynomial}
The $n$th order Bell polynomial is defined as (see Comtet (1974))
\begin{equation}\label{bell}
	\mathcal{B}_n(x)=e^{-x}\sum_{r=0}^{\infty}\frac{r^nx^r}{r!}, \ n\ge 0.
\end{equation}
\subsection{Mittag-Leffler function}
The three-parameter Mittag-Leffler function is defined as (see Kilbas {\it et al.} (2006), p. 45)
\begin{equation}\label{mitag}
	E_{\alpha,\beta}^{\delta}(x)=\frac{1}{\Gamma(\delta)}\sum_{j=0}^{\infty} \frac{\Gamma(j+\delta)x^{j}}{j!\Gamma(j\alpha+\beta)},\ \ x\in\mathbb{R},
\end{equation}
where $\alpha>0$, $\beta>0$ and $\delta>0$. 

It reduces to two-parameter Mittag-Leffler function for $\delta=1$. Further, it reduces to the Mittag-Leffler function for  $\delta=\beta=1$. 

Let $g(t)=t^{\beta-1}E^{\delta}_{\alpha,\beta}(xt^{\alpha})$. Then, for $x\in\mathbb{R}$, the following result holds true (see Kilbas {\it et al.} (2006), Eq. (1.9.13)):
\begin{equation}\label{mi}
	\tilde{g}(s)=\frac{s^{\alpha\delta-\beta}}{(s^{\alpha}-x)^{\delta}},\ s>|x|^{1/\alpha},
\end{equation} 
where $\tilde{g}(s)$ denotes the Laplace transform of the function $g(t)$. 
\subsection{Caputo fractional derivative} The Caputo fractional derivative of the function $f(t)$ is defined as (see Kilbas {\it et al.} (2006))
\begin{equation}\label{caputo}
	\frac{\mathrm{d}^{\alpha}}{\mathrm{d}t^{\alpha}}f(t)=\left\{
	\begin{array}{ll}
		\dfrac{1}{\Gamma{(1-\alpha)}}\displaystyle\int^t_{0} (t-s)^{-\alpha}f'(s)\,\mathrm{d}s,\  0<\alpha<1,\vspace{.2cm}\\	f'(t),\ \alpha=1.
	\end{array}
	\right.
\end{equation}
Its Laplace transform is given by (see Kilbas {\it et al.} (2006), Eq. (5.3.3))
\begin{equation*}
	\tilde{h}(s)=s^{\alpha}\tilde{f}(s)-s^{\alpha-1}f(0),\ \ s>0,
\end{equation*}
where $h(t)=\frac{\mathrm{d}^{\alpha}}{\mathrm{d}t^{\alpha}}f(t)$.
\subsection{Inverse $\alpha$-stable subordinator} A $\alpha$-stable subordinator $\{D^\alpha(t)\}_{t\geq0}$, $0<\alpha<1$ is a non-decreasing L\'evy process. Its Laplace transform is given by $\mathbb{E}(e^{-sD^\alpha(t)})=e^{-ts^\alpha},\, s>0$. Its first passage time $\{Y^\alpha(t)\}_{t\ge0}$ is known as the inverse $\alpha$-stable subordinator and it is defined as
\begin{equation*}
Y^\alpha(t)\coloneqq\inf\{x>0:D^\alpha(x)>t\}.
\end{equation*} 
Its mean and variance are given by (see Leonenko \textit{et al.} (2014))
\begin{equation}\label{meanvarinv}
\mathbb{E}(Y^\alpha(t))= \frac{t^\alpha}{\Gamma(1+\alpha)}\ \ \text{and}\ \ \operatorname{Var}(Y^\alpha(t))=t^{2\alpha}\Big(\frac{2}{\Gamma(2\alpha+1)}-\frac{1}{\Gamma^2(\alpha+1)}\Big),
\end{equation}
respectively.

For fixed $s$ and large $t$, the following asymptotic result holds (see Kataria and Khandakar (2022), Eq. (11)):
\begin{equation}\label{covinv}
\operatorname{Cov}(Y^\alpha(s),Y^\alpha(t))\sim\frac{1}{\Gamma^2(\alpha+1)}\bigg(\alpha s^{2\alpha}B(\alpha, \alpha+1)-\frac{\alpha^2}{(\alpha+1)}\frac{s^{\alpha+1}}{t^{1-\alpha}}\bigg),
\end{equation}
where $B(\alpha,\alpha+1)$ denotes the beta function.

\subsection{Generalized counting process and its compound version}
Here, we give some known results for the GCP and its compound version (see Di Crescenzo {\it et al.} (2016) and Kataria and Khandakar (2022)). 

For each $n\geq0$, the state probability $p(n,t)=\mathrm{Pr}\{M(t)=n\}$ of GCP is given by
\begin{equation}\label{p(n,t)}
	p(n,t)=\sum_{\Omega(k,n)}\prod_{j=1}^{k}\frac{(\lambda_jt)^{x_j}}{x_j!}e^{-\lambda_j t},
\end{equation}
where $\Omega(k,n)=\{(x_1,x_2,\ldots,x_k):\sum_{j=1}^{k}jx_j=n,\ x_j\in \mathbb{N}_0\}$. 

Its state probabilities $p(n,t)$ satisfy the following system of differential equations:
\begin{equation*}
	\frac{\mathrm{d}}{\mathrm{d}t}p(n,t)=-\sum_{j=1}^{k}\lambda_j p(n,t)+\sum_{j=1}^{\min\{n,k\}}\lambda_j p(n-j,t), \ n\geq0,
\end{equation*}
with initial condition $p(n,0)=\delta_n(0)$. 

The L\'evy measure of GCP is given by $\Pi(\mathrm{d}x)=\sum_{j=1}^{k}\lambda_j\delta_j(\mathrm{d}x)$, where $\delta_j$'s are Dirac measures. Its pgf and moment generating function are given by
\begin{equation}\label{pgfmt}
	G(u,t)=\mathbb{E}\left(u^{M(t)}\right)=\exp\bigg(-t\sum_{j=1}^{k}\lambda_{j}(1-u^{j})\bigg),\   |u|\le 1.
\end{equation}
and 
\begin{equation}\label{mgfmt}
	\mathbb{E}\left(e^{uM(t)}\right)=\exp\bigg(-t\sum_{j=1}^{k}\lambda_{j}(1-e^{uj})\bigg),\  u\in \mathbb{R},
\end{equation}
respectively. Also, its mean and variance are 
\begin{equation}\label{covgcp}
\mathbb{E}(M(t))=\sum_{j=1}^{k}j\lambda_jt\ \ \text{and} \ \  \operatorname{Var}(M(t))=\sum_{j=1}^{k}j^2\lambda_jt,
\end{equation}
respectively.

\section{Iterated generalized counting process}\label{secigcp}
In this section, we introduce and study a counting process that is formed by the composition of two independent GCPs. We call it the iterated generalized counting process (IGCP) and denote it by $\{\hat{M}(t)\}_{t\ge0}$. It is defined as 
\begin{equation}\label{IGCPrep}
\hat{M}(t)\coloneqq M(M_0(t)), \, t\ge0,
\end{equation} 
where $\{M_0(t)\}_{t\ge0}$ and $\{M(t)\}_{t\ge0}$ are independent GCPs with positive rates $\mu_{1}$, $\mu_2$, $\dots$, $\mu_{k_0}$ and $\lambda_{1}$, $\lambda_2$, $\dots$, $\lambda_k$, respectively. 

Let $\Omega(k,m)=\{(x_1,x_2,\dots,x_{k}):x_1+2x_2+\dots+kx_k=m,\,x_{j}\in\mathbb{N}_0\}$,  $\lambda=\lambda_{1}+\lambda_2+\dots+\lambda_k$, $\mu=\mu_1+\mu_2+\dots+\mu_{k_0}$ and $z_k=x_1+x_2+\dots+x_k$.
 
In an infinitesimal time interval of length $h$ such that $o(h)/h\to0$ as $h\to0$, the transition probabilities of IGCP are given by
{\footnotesize\begin{equation}\label{transigcp}
\mathrm{Pr}\{\hat{M}(t+h)=n+m|\hat{M}(t)=n\}=\begin{cases}
	1-h\mu+h\sum_{j_0=1}^{k_0}\mu_{j_0}e^{-j_0\lambda}+o(h),\  m=0,\vspace{.2cm}\\	
	h\sum_{j_0=1}^{k_0}\mu_{j_0}\displaystyle\sum_{\Omega(k,m)}\prod_{j=1}^{k}\frac{(j_0\lambda_j)^{x_j}}{x_j!}e^{-j_0\lambda_j}+o(h),\  m>0.
\end{cases}
\end{equation}}
\begin{remark}
For $k=k_0=1$, the IGCP reduces to the iterated Poisson process studied by Orsingher and Polito (2012b).
\end{remark}
For $|u|\le1$, the pgf of IGCP can be obtained as follows:
\begin{align}
\hat{G}(u,t)=\mathbb{E}(u^{\hat{M}(t)})&=\mathbb{E}(\mathbb{E}(u^{M(M_0(t))}|M_0(t)))\nonumber\\
&=\mathbb{E}\Big(\exp\Big(-M_0(t)\sum_{j=1}^{k}\lambda_{j}(1-u^{j})\Big)\Big)\nonumber\\
&=\exp\Big(-\sum_{j_0=1}^{k_0}\mu_{j_0}t\Big(1-\exp\Big(-j_0\sum_{j=1}^{k}\lambda_{j}(1-u^{j})\Big)\Big)\Big).\label{IGDCpgf}
\end{align}
Thus, the governing system of differential equations for $\hat{G}(u,t)$ is given by
\begin{equation}\label{Ipgfde}
\frac{\partial}{\partial t}\hat{G}(u,t)=-\sum_{j_0=1}^{k_0}\mu_{j_0}\Big(1-\exp\Big(-j_0\sum_{j=1}^{k}\lambda_{j}(1-u^{j})\Big)\Big)\hat{G}(u,t),\ \hat{G}(u,0)=1. 
\end{equation}
Similarly, its moment generating function can be obtained in the following form:
\begin{equation}\label{mgfIGCP}
\mathbb{E}(e^{u\hat{M}(t)})=\exp\Big(-\sum_{j_0=1}^{k_0}\mu_{j_0}t\Big(1-\exp\Big(-j_0\sum_{j=1}^{k}\lambda_{j}(1-e^{uj})\Big)\Big)\Big), \ \ u\in\mathbb{R}. 
\end{equation}

\begin{remark}
By substituting $k=k_0=1$ into \eqref{IGDCpgf}, the pgf of IGCP reduces to that of iterated Poisson process given in Eq. (22) of Orsingher and Polito (2012b). 
\end{remark}
 Dhillon and Kataria (2024) showed that the GCP has a unique representation as the weighted sum of independent Poisson processes, that is,
 \begin{equation*}
 M(t)=\sum_{j=1}^{k}jN_j(t),\ t\ge0,
 \end{equation*} 
where $\{N_j(t)\}_{t\ge0}$, $j=1,2,\dots,k$ are independent Poisson processes with positive rates $\lambda_j$, respectively. So, the following result holds true for IGCP: 
\begin{equation*}
\hat{M}(t)=\sum_{j=1}^{k}jN_{j}(M_0(t)),
\end{equation*}
where $\{M_0(t)\}_{t\ge0}$ is independent of $\{N_{j}(t)\}_{t\ge0}$, $j=1,2,\dots,k$.

 In the following result, we show that the IGCP is equal in distribution to a compound GCP.

\begin{proposition}\label{prpcgcp}
Let $X_1$, $X_2$, $\dots$ be independent and identically distributed (iid) random variables, such that, $X_1\overset{d}{=}M(1)$. Then,
\begin{equation*}
\hat{M}(t)\overset{d}{=}\sum_{n=1}^{M_0(t)}X_n,\ t\ge0.
\end{equation*}
\end{proposition}
\begin{proof}
Let $Y(t)=X_1+X_2+\dots+X_{M_0(t)}$. Then, its pgf can be obtained as follows:
\begin{align*}
G_Y(u,t)&=\mathbb{E}(\mathbb{E}(u^{Y(t)}|M_0(t)))\\
&=\mathbb{E}\Big(\prod_{n=1}^{M_0(t)}G_{X_n}(u,t)\Big),\, \text{(as $X_n$'s are independent)}\\
&=\mathbb{E}\Big(\Big(\exp\Big(-\sum_{j=1}^{k}\lambda_{j}(1-u^{j})\Big)\Big)^{M_0(t)}\Big),\ \text{(using \eqref{pgfmt})}\\
&=\exp\Big(-\sum_{j_0=1}^{k_0}\mu_{j_0}t\Big(1-\exp\Big(-j_0\sum_{j=1}^{k}\lambda_{j}(1-u^{j})\Big)\Big)\Big),\ |u|\le1
\end{align*}
which coincides with \eqref{IGDCpgf}. This completes the proof.
\end{proof}
\begin{remark}
For $k=k_0=1$, the result in Proposition \ref{prpcgcp} reduces to that of iterated Poisson process (see Orsingher and Polito (2012b), Eq. (23)).
\end{remark}
\begin{proposition}\label{prpde}
The state probabilities $\hat{p}(n,t)=\mathrm{Pr}\{\hat{M}(t)=n\}$, $n\ge0$ of IGCP satisfy the following system of difference-differential equations:
\begin{equation}\label{IGCPDE}
\frac{\mathrm{d}}{\mathrm{d}t}\hat{p}(n,t)=-\mu\hat{p}(n,t)+\sum_{j_0=1}^{k_0}\mu_{j_0}e^{-j_0\lambda}\sum_{m=0}^{n}\sum_{\Omega(k,m)}\Big(\prod_{j=1}^{k}\frac{(j_0\lambda_{j})^{x_{j}}}{x_{j}!}\Big)\hat{p}(n-m,t)  
\end{equation}
with initial condition $\hat{p}(n,0)=\delta_n(0)$. 
\end{proposition}
\begin{proof}
Note that
\begin{align*}
\hat{p}(n,t+h)&=\hat{p}(n,t)\hat{p}(0,h)+\sum_{m=1}^{n}\hat{p}(n-m,t)\hat{p}(m,h)+o(h).
\end{align*}
By using \eqref{transigcp}, we get
\begin{align*}
\frac{\hat{p}(n,t+h)-\hat{p}(n,t)}{h}&=-\mu \hat{p}(n,t)+\sum_{j_0=1}^{k_0}\mu_{j_0}\sum_{m=0}^{n}\sum_{\Omega(k,m)}\Big(\prod_{j=1}^{k}\frac{(j_0\lambda_j)^{x_j}}{x_j!}e^{-j_0\lambda_j}\Big)\hat{p}(n-m,t)+\frac{o(h)}{h}.
\end{align*}
On letting $h\to0$, we get the required result.
For an alternate proof, we refer the reader to Appendix A1.
\end{proof}
\begin{remark}
For $k=k_0=1$, the result in Proposition \ref{prpde} reduces to that of iterated Poisson process (see Orsingher and Polito (2012b), Eq. (27)).
\end{remark}
Next, we obtain the state probabilities of IGCP in terms of Bell polynomials.
\begin{theorem}\label{thmpmf}
The state probabilities of IGCP are given by
\begin{equation}\label{IGpmf}
\hat{p}(n,t)=\sum_{\Omega(k,n)}\Big(\prod_{j=1}^{k}\frac{\lambda_{j}^{n_{j}}}{n_{j}!}\Big)\sum_{\sum_{j_0=1}^{k_0}r_{j_0}=z_{k}}z_{k}!\prod_{j_0=1}^{k_0}\frac{j_0^{r_{j_0}}}{r_{j_0}!}e^{-\mu_{j_0}t(1-e^{-j_0\lambda})}\mathcal{B}_{r_{j_0}}(e^{-j_0\lambda}\mu_{j_0}t),\ n\ge0,
\end{equation}
where $\mathcal{B}_n(x)$ is the $n$th order Bell polynomial defined in \eqref{bell}.
\end{theorem}
\begin{proof}
From \eqref{IGCPrep}, we have
\begin{align}
\hat{p}(n,t)&=\sum_{s=0}^{\infty}\mathrm{Pr}\{M(s)=n\}\mathrm{Pr}\{M_0(t)=s\}\nonumber\\
%&=\sum_{s=0}^{\infty}\sum_{\Omega(k,n)}\Big(\prod_{j=1}^{k}\frac{(\lambda_{j}s)^{n_{j}}}{n_{j}!}e^{-\lambda_{j}s}\Big)\sum_{\Omega(k_0,s)}\prod_{j_0=1}^{k_0}\frac{(\mu_{j_0}t)^{x_{j_0}}}{x_{j_0}!}e^{-\mu_{j_0}t}\\
&=\sum_{\Omega(k,n)}\Big(\prod_{j=1}^{k}\frac{\lambda_{j}^{n_{j}}}{n_{j}!}\Big)e^{-\mu t}\sum_{s=0}^{\infty}s^{z_{k}}e^{-\lambda s}\sum_{\Omega(k_0,s)}\prod_{j_0=1}^{k_0}\frac{(\mu_{j_0}t)^{x_{j_0}}}{x_{j_0}!}\label{igcppmfep}\\
&=\sum_{\Omega(k,n)}\Big(\prod_{j=1}^{k}\frac{\lambda_{j}^{n_{j}}}{n_{j}!}\Big)e^{-\mu t}\sum_{\substack{x_{j_0}\ge0\\1\le j_0\le k_0}}\Big(\sum_{j_0=1}^{k_0}j_0x_{j_0}\Big)^{z_{k}}\prod_{j_0=1}^{k_0}\frac{(e^{-j_0\lambda}\mu_{j_0}t)^{x_{j_0}}}{x_{j_0}!}\nonumber\\
&=\sum_{\Omega(k,n)}\Big(\prod_{j=1}^{k}\frac{\lambda_{j}^{n_{j}}}{n_{j}!}\Big)e^{-\mu t}\sum_{\sum_{j_0=1}^{k_0}r_{j_0}=z_{k}}z_{k}!\Big(\prod_{j_0=1}^{k_0}\frac{j_0^{r_{j_0}}}{r_{j_0}!}\Big)\sum_{\substack{x_{j_0}\ge0\\1\le j_0\le k_0}}\prod_{j_0=1}^{k_0}\frac{x_{j_0}^{r_{j_0}}(e^{-j_0\lambda}\mu_{j_0}t)^{x_{j_0}}}{x_{j_0}!}.\nonumber
%&=\sum_{\Omega(k,n)}\Big(\prod_{j=1}^{k}\frac{\lambda_{j}^{n_{j}}}{n_{j}!}\Big)\sum_{\sum_{j_0=1}^{k_0}r_{j_0}=z_{k}}z_{k}!\prod_{j_0=1}^{k_0}\frac{j_0^{r_{j_0}}}{r_{j_0}!}e^{-\mu_{j_0}t(1-e^{-j_0\lambda})}\mathcal{B}_{r_{j_0}}(e^{-j_0\lambda}\mu_{j_0}t).
\end{align}
By using \eqref{bell}, we get the required result.
\end{proof}
\begin{remark}
On taking $k=k_0=1$ in \eqref{IGpmf}, we get the state probabilities of iterated Poisson process (see Orsingher and Polito (2012b), Theorem 2.1). 
\end{remark}
Note that the IGCP is a L\'evy process as it is the composition of two independent L\'evy processes.  Its L\'evy measure is obtained in the following result:
\begin{proposition}
The L\'evy measure of IGCP is given by
\begin{equation*}
\Pi_{\hat{M}}(\mathrm{d}x)=\sum_{j_0=1}^{k_0}\mu_{j_0}\sum_{n=1}^{\infty}\sum_{\Omega(k,n)}\Big(\prod_{j=1}^{k}\frac{(j_0\lambda_{j})^{n_{j}}}{n_{j}!}e^{-j_0\lambda_{j}}\Big)\delta_n(\mathrm{d}x),
\end{equation*}
where $\delta_n$'s are Dirac measures.
\end{proposition}
\begin{proof}
	Let $p(n,s)$ and $\Pi_{M_0}(\mathrm{d}s)$ be the pmf and L\'evy measure of $\{M(t)\}_{t\ge0}$ and $\{M_0(t)\}_{t\ge0}$, respectively. By using Theorem 30.1 of Sato (1999), we have
	\begin{align*}
		\Pi_{\hat{M}}(\mathrm{d}x)&=\int_{s=0}^{\infty}\sum_{n=1}^{\infty}p(n,s)\,\delta_n(\mathrm{d}x)\,\Pi_{M_0}(\mathrm{d}s)\\
	&=\int_{s=0}^{\infty}\sum_{n=1}^{\infty}\sum_{\Omega(k,n)}\Big(\prod_{j=1}^{k}\frac{(\lambda_{j}s)^{n_{j}}}{n_{j}!}e^{-\lambda_{j}s}\Big)\delta_n(\mathrm{d}x)\sum_{j_0=1}^{k_0}\mu_{j_0}\delta_{j_0}(\mathrm{d}s)
%		&=\sum_{j_0=1}^{k_0}\mu_{j_0}\sum_{n=1}^{\infty}\sum_{\Omega(k,n)}\Big(\prod_{j=1}^{k}\frac{(j_0\lambda_{j})^{n_{j}}}{n_{j}!}e^{-j_0\lambda_{j}}\Big)\delta_n(\mathrm{d}x),
	\end{align*}
which reduces to the required result.
\end{proof}
Let $S=\sum_{j=1}^{k}j\lambda_j\sum_{j_0=1}^{k_0}j_0\mu_{j_0}$ and $T=\Big(\sum_{j=1}^{k}j\lambda_j\Big)^2\sum_{j_0=1}^{k_0}j_0^2\mu_{j_0}+\sum_{j=1}^{k}j^2\lambda_j\sum_{j_0=1}^{k_0}j_0\mu_{j_0}$. By using \eqref{covgcp} and Theorem 2.1 of Leonenko \textit{et al.} (2014), the mean, variance and covariance of IGCP are given by
\begin{align}
\mathbb{E}(\hat{M}(t))&=\mathbb{E}(M(1))\mathbb{E}(M_0(t))=S t,\label{IGCPMEAN}\\
\operatorname{Var}(\hat{M}(t))&=(\mathbb{E}(M(1)))^2\operatorname{Var}(M_0(t))+\operatorname{Var}(M(1))\mathbb{E}(M_0(t))=T t.\label{IGCPVAR}\\
\operatorname{Cov}(\hat{M}(s),\hat{M}(t))&=\operatorname{Var}(\hat{M}(s))=Ts, \ 0<s\le t.\label{IGCPCov}
\end{align}
As $\operatorname{Var}(\hat{M}(t))-\mathbb{E}(\hat{M}(t))>0$, the IGCP is overdispersed.

\begin{theorem}\label{thmhit}
	Let $T_n\coloneqq\inf\{s>0:\hat{M}(s)=n\}$ be the first-passage time of IGCP for any state $n\ge1$. Then, its distribution is given by
	\begin{align*}
		\mathrm{Pr}\{T_n\in\mathrm{d}s\}&=\sum_{j_0=1}^{k_0}\mu_{j_0}e^{-s\mu}\mathrm{d}s\sum_{m=1}^{n}\sum_{r=0}^{\infty}\sum_{\Omega(k,n-m)}\Big(\prod_{j=1}^{k}\frac{(r\lambda_{j})^{x_{j}}}{x_{j}!}\Big)e^{-r\lambda}\\ 
		&\hspace{4cm}\cdot  \sum_{\Omega(k,m)}\Big(\prod_{j=1}^{k}\frac{(j_0\lambda_{j})^{n_{j}}}{n_{j}!}\Big)e^{-j_0\lambda}\sum_{\Omega(k_0,r)}\prod_{j_0=1}^{k_0}\frac{(s\mu_{j_0})^{n_{j_0}}}{n_{j_0}!}.
	\end{align*}
\end{theorem}
\begin{proof}
	We have
	\begin{align}
		\mathrm{Pr}\{T_n\in\mathrm{d}s\}&=\sum_{m=1}^{n}\mathrm{Pr}\{M(M_0(s))=n-m,M(M_0(s+\mathrm{d}s))=n\}\nonumber\\
		&=\sum_{m=1}^{n}\mathrm{Pr}\{M(M_0(s))=n-m,M(M_0(s)+\mathrm{d}M_0(s))=n\}\nonumber\\
		&=\sum_{j_0=1}^{k_0}\mu_{j_0}\mathrm{d}s\sum_{m=1}^{n}\mathrm{Pr}\{M(M_0(s))=n-m,M(M_0(s)+j_0)=n\}.\label{hitA}
	\end{align}
	Observe that the process $\{M(t)\}_{t\ge0}$ performs jumps of size $0\le n-m\le n-1$ in time interval $(M_0(s), M_0(s)+j_0)$. Also, 
	\begin{align}
		\mathrm{Pr}\{M(M_0(s))&=n-m,M(M_0(s)+j_0)=n\}\nonumber\\
		&=\sum_{r=0}^{\infty}\mathrm{Pr}\{M(r)=n-m,M(r+j_0)=n\}\mathrm{Pr}\{M_0(s)=r\}\nonumber\\
		&=\sum_{r=0}^{\infty}\sum_{\Omega(k,n-m)}\Big(\prod_{j=1}^{k}\frac{(r\lambda_{j})^{x_{j}}}{x_{j}!}\Big)e^{-r\lambda}\nonumber\\ 
		&\hspace{2cm}\cdot  \sum_{\Omega(k,m)}\Big(\prod_{j=1}^{k}\frac{(j_0\lambda_{j_0})^{n_{j}}}{n_{j}!}\Big)e^{-j_0\lambda}\sum_{\Omega(k_0,r)}\prod_{j_0=1}^{k_0}\frac{(s\mu_{j_0})^{n_{j_0}}}{n_{j_0}!}e^{-s\mu_{j_0}}.\label{hitB}
	\end{align}
	By using \eqref{hitB} and \eqref{hitA}, we get the required result.
\end{proof}
\begin{corollary}
For $n=1$, we have
\begin{equation*}
	\mathrm{Pr}\{T_1\in\mathrm{d}s\}=\sum_{j_0=1}^{k_0}\mu_{j_0}\mathrm{d}s\sum_{r=0}^{\infty}j_0\lambda_1e^{-j_0\lambda_1}e^{-r\lambda}\sum_{\Omega(k_0,r)}\prod_{j_0=1}^{k_0}\frac{(s\mu_{j_0})^{n_{j_0}}}{n_{j_0}!}e^{-s\mu_{j_0}}.
\end{equation*}
Thus, 
\begin{equation*}
	\mathrm{Pr}\{T_1<\infty\}=\sum_{j_0=1}^{k_0}\frac{\mu_{j_0}}{\mu}j_0\lambda_1e^{-j_0\lambda_1}\sum_{r=0}^{\infty}e^{-r\lambda}\sum_{\Omega(k_0,r)}z_{k_0}!\prod_{j_0=1}^{k_0}\Big(\frac{\mu_{j_0}}{\mu}\Big)^{n_{j_0}}\frac{1}{n_{j_0}!}.
\end{equation*}

\end{corollary}

\begin{remark}
	For $k=k_0=1$, the result in Theorem \ref{thmhit} reduces to that of iterated Poisson process (see Orsingher and Polito (2012b), Theorem 2.2).
\end{remark}

\subsection{Martingale characterization for IGCP}
Here, we give a martingale characterization and related results for the IGCP. First, we observe that the converse part of result in  Proposition 2 of Kataria and Khandakar (2022) holds true.
\begin{proposition}
The process $\{M(t)\}_{t\ge0}$ is a GCP with positive rates $\lambda_j$, $j=1,2,\dots,k$ iff $\{M(t)-\sum_{j=1}^{k}j\lambda_j t\}_{t\ge0}$ is a $\{\mathcal{F}_{t}\}_{t\ge0}$-martingale, where $\mathcal{F}_t=\sigma(M(s),\, s\le t)$.
\end{proposition} 
\begin{proof}
Let $\{N(t)\}_{t\ge0}$ be a Poisson process with rate $\lambda=\lambda_1+\lambda_2+\dots+\lambda_k$. For $\nu=1$ in Eq. (2.6) of Di Crescenzo \textit{et al.} (2016), we have 
\begin{equation*}
M(t)\overset{d}{=}\sum_{i=1}^{N(t)}X_i,
\end{equation*}
where $X_i$'s are iid random variables such that $\mathrm{Pr}\{X_1=i\}=\lambda_i/\lambda$, $i=1,2,\dots,k$ and are independent of $\{N(t)\}_{t\ge0}$. That is, the GCP is equal in distribution to a compound Poisson process. The result follows from Theorem 5.2 of Zhang and Li (2016).
\end{proof}
Next, we give a martingale characterization for the IGCP.
	\begin{proposition}\label{prpexpmart}
		Let $\{\hat{M}(t)\}_{t\ge0}$ be a point process such that $\hat{M}(0)=0$. Then, $\{\hat{M}(t)\}_{t\ge0}$ is IGCP iff the process
		\begin{equation*}
		Y(t)=\exp\Big(u\hat{M}(t)-\sum_{j_0=1}^{k_0}\mu_{j_0}t\Big(\exp\Big(-j_0\sum_{j=1}^{k}\lambda_{j}(1-e^{uj})\Big)-1\Big)\Big), \ u\in\mathbb{R},\  t\ge0
		\end{equation*}  
is a $\{\mathcal{F}_t\}_{t\ge0}$-martingale, where $\mathcal{F}_t=\sigma(\hat{M}(s),\, s\le t)$.
	\end{proposition}
	\begin{proof}
		Let $\{\hat{M}(t)\}_{t\ge0}$ be IGCP. Then, we just need to show the martingale property for $Y(t)=\exp\Big(u\hat{M}(t)-\sum_{j_0=1}^{k_0}\mu_{j_0}t\Big(\exp\Big(-j_0\sum_{j=1}^{k}\lambda_{j}(1-e^{uj})\Big)-1\Big)\Big)$. For $0<s\le t$, we have
		\begin{align*}
			\mathbb{E}(Y(t)|\mathcal{F}_s)&=\mathbb{E}(e^{u(\hat{M}(t)-\hat{M}(s))}|\mathcal{F}_s)\exp\Big(u\hat{M}(s)-\sum_{j_0=1}^{k_0}\mu_{j_0}t\Big(\exp\Big(-j_0\sum_{j=1}^{k}\lambda_{j}(1-e^{uj})\Big)-1\Big)\Big)\\
			&=\mathbb{E}(e^{u(\hat{M}(t-s))})\exp\Big(u\hat{M}(s)-\sum_{j_0=1}^{k_0}\mu_{j_0}t\Big(\exp\Big(-j_0\sum_{j=1}^{k}\lambda_{j}(1-e^{uj})\Big)-1\Big)\Big)\\
			&=\exp\Big(u\hat{M}(s)-\sum_{j_0=1}^{k_0}\mu_{j_0}s\Big(\exp\Big(-j_0\sum_{j=1}^{k}\lambda_{j}(1-e^{uj})\Big)-1\Big)\Big),
		\end{align*}
		where we have used the independent and stationary increments of $\{\hat{M}(t)\}_{t\ge0}$.
		
		Conversely, let $\{Y(t)\}_{t\ge0}$ be a $\{\mathcal{F}_t\}_{t\ge0}$-martingale. Then, we have $\mathbb{E}(Y(t))=\mathbb{E}(Y(0))=1$ for all $t\ge0$. Thus,
		\begin{equation*}
			\mathbb{E}(e^{u\hat{M}(t)})=\exp\Big(\sum_{j_0=1}^{k_0}\mu_{j_0}t\Big(\exp\Big(-j_0\sum_{j=1}^{k}\lambda_j(1-e^{uj})\Big)-1\Big)\Big).
		\end{equation*}
		From \eqref{mgfIGCP}, it follows that $\{\hat{M}(t)\}_{t\ge0}$ has same distribution to that of the IGCP.
		
		For $0<s\le t$, by using the martingale property of $\{Y(t)\}_{t\ge0}$, we have
		\begin{align*}
			\mathbb{E}\Big(\exp\Big(u\hat{M}(t)-\sum_{j_0=1}^{k_0}\mu_{j_0}t\Big(&\exp\Big(-j_0\sum_{j=1}^{k}\lambda_{j}(1-e^{uj})\Big)-1\Big)\Big)\Big|\mathcal{F}_s\Big)\\
			&=\exp\Big(u\hat{M}(s)-\sum_{j_0=1}^{k_0}\mu_{j_0}s\Big(\exp\Big(-j_0\sum_{j=1}^{k}\lambda_{j}(1-e^{uj})\Big)-1\Big)\Big).
		\end{align*}
		So, 
		\begin{equation}\label{eqstat}
			\mathbb{E}(e^{u(\hat{M}(t)-\hat{M}(s))}|\mathcal{F}_s)=\exp\Big(\sum_{j_0=1}^{k_0}\mu_{j_0}(t-s)\Big(\exp\Big(-j_0\sum_{j=1}^{k}\lambda_{j}(1-e^{uj})\Big)-1\Big)\Big).
		\end{equation}
		Now, by taking expectation on both sides of \eqref{eqstat}, we get
		\begin{equation*}
			\mathbb{E}(e^{u(\hat{M}(t)-\hat{M}(s))})=\exp\Big(\sum_{j_0=1}^{k_0}\mu_{j_0}(t-s)\Big(\exp\Big(-j_0\sum_{j=1}^{k}\lambda_{j}(1-e^{uj})\Big)-1\Big)\Big).
		\end{equation*}
		This establishes that $\{\hat{M}(t)\}_{t\ge0}$ exhibits stationary increments, that is, $\hat{M}(t)-\hat{M}(s)\overset{d}{=}\hat{M}(t-s)$.
		
		For $0\le t_0<t_1<\dots<t_n<\infty$, we have
		\begin{align*}
			\mathbb{E}(e^{\sum_{i=1}^{n}u_i(\hat{M}(t_i)-\hat{M}(t_{i-1}))})&=\mathbb{E}(\mathbb{E}(e^{\sum_{i=1}^{n}u_i(\hat{M}(t_i)-\hat{M}(t_{i-1}))}|\mathcal{F}_{t_{n-1}}))\\
			&=\mathbb{E}(e^{\sum_{i=1}^{n-1}u_i(\hat{M}(t_i)-\hat{M}(t_{i-1}))})\mathbb{E}(e^{u_n(\hat{M}(t_n)-\hat{M}(t_{n-1}))})\\
			&\ \, \vdots\\
			&=\prod_{i=1}^{n}\mathbb{E}(e^{u_i(\hat{M}(t_i)-\hat{M}(t_{i-1}))}).
		\end{align*}
		 Thus, $\{\hat{M}(t)\}_{t\ge0}$ has independent increments. Hence, $\{\hat{M}(t)\}_{t\ge0}$ is an IGCP. This completes the proof.
	\end{proof}

\begin{proposition}\label{prepmar}
	The process $\{\hat{M}(t)\}_{t\ge0}$ is IGCP iff $\{\hat{M}(t)-\sum_{j=1}^{k}j\lambda_{j}\sum_{j_0=1}^{k_0}j_0\mu_{j_0}t\}_{t\ge0}$ is a martingale with respect to natural filtration $\mathcal{F}_t=\sigma(\hat{M}(s),\, 0<s\le t)$. 
\end{proposition}
\begin{proof}
By using Proposition \ref{prpcgcp}, the IGCP is equal in distribution to a compound GCP. Further, by using Remark 4.2 of Khandakar and Kataria (2024), the IGCP is equal in distribution to a compound Poisson process. Finally, the result follows on using Theorem 5.2 of Zhang and Li (2016).
	This completes the proof.
\end{proof}
	\begin{remark}
	On substituting $k=k_0=1$ in Proposition \ref{prpexpmart} and Proposition \ref{prepmar}, we obtain the corresponding results for iterated Poisson process. 
\end{remark}

	\subsection{Fractional integral of the IGCP}
Here, we study the Riemann-Liouville fractional integral of IGCP defined as follows:
\begin{equation}\label{malphabeta}
	\hat{\mathcal{X}}^{\alpha}(t)\coloneqq\frac{1}{\Gamma(\alpha)}\int_{0}^{t}(t-s)^{\alpha-1}\hat{M}(s)\,\mathrm{d}s, \ \alpha>0, \, t\geq0.
\end{equation}
For $\alpha=m\in\mathbb{N}$, the Riemann-Liouville fractional integral defined in \eqref{malphabeta} reduces to
\begin{equation*}
	\hat{\mathcal{X}}^{m}(t)=	\frac{1}{(m-1)!}\int_{0}^{t}(t-s)^{m-1}\hat{M}(s)\,\mathrm{d}s=\int_{0}^{t}\mathrm{d}s_{1}\int_{0}^{s_{1}}\mathrm{d}s_{2}\cdots\int_{0}^{s_{m-1}}\hat{M}(s_{m})\mathrm{d}s_{m}.
\end{equation*}

Orsingher and Polito (2013) studied the fractional integral of time fractional Poisson process. The motivation to study such integrals lies in the fact that the integrated counting processes often arise in epidemic model, biological sciences, {\it etc.} (see Pollett (2003), Vishwakarma and Kataria (2024a), (2024b), and references therein).

By using \eqref{IGCPMEAN}, the mean of fractional integral of IGCP can be obtained as follows:
\begin{align}
	\mathbb{E}(\hat{\mathcal{X}}^{\alpha}(t))&=\frac{1}{\Gamma(\alpha)}\int_{0}^{t}(t-s)^{\alpha-1}\mathbb{E}(\hat{M}(s))\,\mathrm{d}s\nonumber\\
	&=\frac{S}{\Gamma(\alpha)}\int_{0}^{t}(t-s)^{\alpha-1}s\,\mathrm{d}s=\frac{St^{\alpha+1}}{\Gamma(\alpha+2)}. \label{meanintgcp}
\end{align}
Also, by using \eqref{IGCPMEAN} and \eqref{IGCPCov}, we have
\begin{align}
	\mathbb{E}(\hat{\mathcal{X}}^{\alpha}(t))^2&=\frac{1}{\Gamma^2(\alpha)}\int_{0}^{t}\int_{0}^{t}(t-s)^{\alpha-1}(t-w)^{\alpha-1}\mathbb{E}(\hat{M}(s)\hat{M}(w))\,\mathrm{d}w\,\mathrm{d}s\nonumber\\
	&=\frac{2}{\Gamma^2(\alpha)}\int_{0}^{t}\int_{s}^{t}(t-s)^{\alpha-1}(t-w)^{\alpha-1}(Ts+S^2ws)\mathrm{d}w\,\mathrm{d}s\nonumber\\
	&=\frac{2T}{\Gamma^2(\alpha)}\int_{0}^{t}s(t-s)^{\alpha-1}\,\mathrm{d}s\int_{s}^{t}(t-w)^{\alpha-1}\,\mathrm{d}w \nonumber\\ 
	&\ \ +\frac{2S^{2}}{\Gamma^2(\alpha)}\int_{0}^{t}s(t-s)^{\alpha-1}\,\mathrm{d}s\int_{s}^{t}w(t-w)^{\alpha-1}\,\mathrm{d}w\nonumber\\
	%	&=\frac{2T}{\Gamma^2(\alpha)}\frac{1}{\alpha}\int_{0}^{t}s(t-s)^{2\alpha-1}\,\mathrm{d}s  +\frac{2S^2}{\Gamma^2(\alpha)}\frac{1}{\alpha}\int_{0}^{t}s^2(t-s)^{2\alpha-1}\,\mathrm{d}s \nonumber\\
	%	&\ \ +\frac{2S^2}{\Gamma^2(\alpha)}\frac{1}{\alpha(\alpha+1)}\int_{0}^{t}s(t-s)^{2\alpha}\,\mathrm{d}s\nonumber\\
	&=\frac{Tt^{2\alpha+1}}{(2\alpha+1)\Gamma^2(\alpha+1)}  +\frac{\left(St^{\alpha+1}\right)^2}{\Gamma^2(\alpha+2)}.\label{varintigcp}
\end{align}
Thus, by using \eqref{meanintgcp} and \eqref{varintigcp}, the variance of $\{\hat{\mathcal{M}}^\alpha(t)\}_{t\ge0}$ is given by
\begin{equation*}
	\operatorname{Var}(\hat{\mathcal{X}}^\alpha(t))=\frac{Tt^{2\alpha+1}}{(2\alpha+1)\Gamma^2(\alpha+1)}.
\end{equation*}
\begin{proposition}
	The conditional mean of the fractional integral of IGCP is given by 
	{\begin{align*}
			\mathbb{E}(\hat{\mathcal{X}}^\alpha(t)|\hat{M}(t)=n)&=\frac{1}{\Gamma(\alpha)}\sum_{r=0}^{n}\frac{re^{-\mu t}}{\mathrm{Pr}\{\hat{M}(t)=n\}}\sum_{x=0}^{\infty}\mathrm{Pr}\{M(x)=r\}\sum_{\Omega(k_0,x)}\Big(\prod_{j_0=1}^{k_0}\frac{\mu_{j_0}^{r_{j_0}}}{r_{j_0}!}\Big)\\
			&\ \ \cdot\sum_{y=0}^{\infty}\mathrm{Pr}\{M(y)=n-r\}\sum_{\Omega(k_0,y)}\Big(\prod_{j_0=1}^{k_0}\frac{\mu_{j_0}^{l_{j_0}}}{l_{j_0}!}\Big)t^{\alpha+l+r}B(r+1,\alpha+l),
	\end{align*}}
	where $r=r_1+r_2+\dots+r_{k_0}$, $l=l_1+l_2+\dots+l_{k_0}$ and $B(a,b)$ denotes the beta function.
\end{proposition}
\begin{proof}
	From \eqref{malphabeta}, we have
	\begin{align*}
		\mathbb{E}(\hat{\mathcal{X}}^\alpha&(t)|\hat{M}(t)=n)\\
		&=\frac{1}{\Gamma(\alpha)}\int_{0}^{t}(t-s)^{\alpha-1}\mathbb{E}(\hat{M}(s)|\hat{M}(t)=n)\,\mathrm{d}s\\
		&=\frac{1}{\Gamma(\alpha)}\sum_{r=0}^{n}r\int_{0}^{t}(t-s)^{\alpha-1}\mathrm{Pr}\{\hat{M}(s)=r|\hat{M}(t)=n\}\,\mathrm{d}s\\
		&=\frac{1}{\Gamma(\alpha)}\sum_{r=0}^{n}\frac{r}{\mathrm{Pr}\{\hat{M}(t)=n\}}\int_{0}^{t}(t-s)^{\alpha-1}\mathrm{Pr}\{\hat{M}(s)=r\}\mathrm{Pr}\{\hat{M}(t-s)=n-r\}\,\mathrm{d}s.
	\end{align*}
	The result follows on using \eqref{igcppmfep}.
\end{proof}
For $\alpha=1$, the integral in \eqref{malphabeta} reduces to the following integral:
\begin{equation}
	\hat{\mathcal{X}}(t)=\int_{0}^{t}\hat{M}(s)\,\mathrm{d}s,\ t\ge0.
\end{equation}
Note that the bivariate process $\{\hat{M}(t),\hat{\mathcal{X}}(t)\}_{t\ge0}$ is a Markov process. Its joint distribution $\hat{q}(n,x,t)=\mathrm{Pr}\{\hat{M}(t)=n,\hat{\mathcal{X}}(t)\le x|\hat{M}(0)=0\}$, $n\ge0$, $x\ge0$ solves the following differential equation:
\begin{equation*}
	\frac{\partial}{\partial t}\hat{q}(n,x,t)+n\frac{\partial}{\partial x}\hat{q}(n,x,t)=-\mu \hat{q}(n,x,t)+\sum_{j_0=1}^{k_0}\mu_{j_0}e^{-j_0\lambda}\sum_{m=0}^{n}\sum_{\Omega(k,m)}\Big(\prod_{j=1}^{k}\frac{(j_0\lambda_j)^{x_j}}{x_j!}\Big)\hat{q}(n-m,x,t)
\end{equation*}
with $\hat{q}(0,0,0)=1$.

Also, its joint pgf, that is, 
\begin{align*}
	\hat{G}(u,v,t)&\coloneqq\int_{0}^{\infty}\sum_{n=0}^{\infty}u^nv^x\frac{\partial}{\partial x}\hat{q}(n,x,t)\, \mathrm{d}x\\
	&=-\ln v\sum_{n=0}^{\infty}u^n\int_{0}^{\infty}v^x\hat{q}(n,x,t)\, \mathrm{d}x,\ 0<u<1,\, 0<v<1
\end{align*} 
is the solution of the following differential equation:
\begin{equation*}
	\frac{\partial}{\partial t}\hat{G}(u,v,t)=\Big(-\mu +\sum_{j_0=1}^{k_0}\mu_{j_0}e^{-j_0\lambda}\sum_{m=0}^{\infty}u^m\sum_{\Omega(k,m)}\Big(\prod_{j=1}^{k}\frac{(j_0\lambda_j)^{x_j}}{x_j!}\Big)\Big)\hat{G}(u,v,t)+u\ln v \frac{\partial}{\partial u}\hat{G}(u,v,t)
\end{equation*}
with $\hat{G}(u,v,0)=1$. Thus, the cumulative generating function $\hat{K}(\theta,\phi,t)=\ln \hat{G}(e^\theta,e^\phi,t)$ solves
\begin{equation}\label{cgfint}
	\frac{\partial}{\partial t}\hat{k}(\theta,\phi,t)=-\mu +\sum_{j_0=1}^{k_0}\mu_{j_0}e^{-j_0\lambda}\sum_{m=0}^{\infty}e^{m\theta}\sum_{\Omega(k,m)}\Big(\prod_{j=1}^{k}\frac{(j_0\lambda_j)^{x_j}}{x_j!}\Big)+ \phi \frac{\partial}{\partial \theta}\hat{K}(\theta,\phi,t)
\end{equation}
with $\hat{K}(\theta,\phi,0)=0$.

Now, on substituting the following expression of the cumulative generating function (see Kendall (1948), Eq. (35)):
\begin{equation}
	\hat{K}(\theta,\phi,t)=\theta \mathbb{E}(\hat{M}(t))+\phi\mathbb{E}(\hat{\mathcal{X}}(t))+\frac{\theta^2}{2}\operatorname{Var}(\hat{M}(t))+\frac{\phi^2}{2}\operatorname{Var}(\hat{\mathcal{X}}(t))+\theta\phi\operatorname{Cov}(\hat{M}(t),\hat{\mathcal{X}}(t))+\dots
\end{equation}
in \eqref{cgfint}, we get 
\begin{align}
	&\frac{\partial}{\partial t}\Big(\theta \mathbb{E}(\hat{M}(t))+\phi\mathbb{E}(\hat{\mathcal{X}}(t))+\frac{\theta^2}{2}\operatorname{Var}(\hat{M}(t))+\frac{\phi^2}{2}\operatorname{Var}(\hat{\mathcal{X}}(t))+\theta\phi\operatorname{Cov}(\hat{M}(t),\hat{\mathcal{X}}(t))+\dots\Big)\nonumber\\
	&=\Big(-\mu +\sum_{j_0=1}^{k_0}\mu_{j_0}e^{-j_0\lambda}\sum_{m=0}^{\infty}\sum_{r=0}^{\infty}\frac{(m\theta)^r}{r!}\sum_{\Omega(k,m)}\Big(\prod_{j=1}^{k}\frac{(j_0\lambda_j)^{x_j}}{x_j!}\Big)\Big)+ \phi  \mathbb{E}(\hat{M}(t)+\theta\operatorname{Var}(\hat{M}(t))\nonumber\\
	&\ \ +\phi\operatorname{Cov}(\hat{M}(t),\hat{\mathcal{X}}(t))+\dots.\label{kindel}
\end{align}
On comparing the coefficient of $\theta \phi$ on both sides of \eqref{kindel}, we get
\begin{align*}
	%\frac{\mathrm{d}}{\mathrm{d}t}\operatorname{Var}(\hat{M}(t))&=\sum_{j_0=1}^{k_0}\mu_{j_0}e^{-j_0\lambda}\sum_{m=0}^{\infty}m^2\sum_{\Omega(k,m)}\Big(\prod_{j=1}^{k}\frac{(j_0\lambda_j)^{x_j}}{x_j!}\Big)\\
	\frac{\mathrm{d}}{\mathrm{d}t}\operatorname{Cov}(\hat{M}(t),\hat{\mathcal{X}}(t))&=\operatorname{Var}(\hat{M}(t))=Tt
\end{align*}
with $\operatorname{Cov}(\hat{M}(0),\hat{\mathcal{X}}(0))=0$. 
So,
\begin{equation*}
	\operatorname{Cov}(\hat{M}(t),\hat{\mathcal{X}}(t))=\operatorname{Var}(\hat{M}(t))=\frac{Tt^2}{2}.
\end{equation*}
Here, $T=\Big(\sum_{j=1}^{k}j\lambda_j\Big)^2\sum_{j_0=1}^{k_0}j_0^2\mu_{j_0}+\sum_{j=1}^{k}j^2\lambda_j\sum_{j_0=1}^{k_0}j_0\mu_{j_0}$.

\subsection{Non-homogeneous IGCP}
Here, we study a non-homogeneous version of the IGCP. It is obtained by considering a non-homogeneous version of the time changing component process $\{M_0(t)\}_{t\ge0}$ in \eqref{IGCPrep}. 

Let $\{\mathcal{M}_0(t)\}_{t\ge0}$ be a non-homogeneous GCP with time-dependent rate functions $\mu_{j_0}(t):[0,\infty)\to [0,\infty)$, $j_0=1,2,\dots,k_0$. We consider the following time-changed process:
\begin{equation*}
\hat{\mathcal{M}}(t)\coloneqq M(\mathcal{M}_0(t)),\, t\ge0,
\end{equation*}
where $\{\mathcal{M}_0(t)\}_{t\ge0}$ is independent of the GCP $\{M(t)\}_{t\ge0}$. 

By using the conditional argument, the pgf of $\{\hat{\mathcal{M}}(t)\}_{t\ge0}$ can be obtained in the following form:
\begin{equation*}
\mathbb{E}(u^{\hat{\mathcal{M}}(t)})=\exp\Big(-\sum_{j_0=1}^{k_0}\rho_{j_0}(t)\Big(1-\exp\Big(-j_0\sum_{j=1}^{k}\lambda_{j}(1-u^{j})\Big)\Big)\Big), \ |u|\le1,
\end{equation*} 
where $\rho_{j_0}(t)=\int_{0}^{t}\mu_{j_0}(s)\,\mathrm{d}s$ is the cumulative rate function. Also, 
\begin{equation}\label{eqngcpcomp}
\hat{\mathcal{M}}(t)\overset{d}{=}\sum_{i=1}^{\mathcal{M}_0(t)}X_i,\, t\ge0,
\end{equation} 
where $X_i$'s are iid random variables which are equal in distribution to that of $M(1)$. 
The proof of \eqref{eqngcpcomp} follows similar lines to that of Proposition \ref{prpcgcp}.
%{\color{blue}\begin{remark}
%Note that, if $\{X_i\}_{i\ge1}$ are iid random variables with some distribution function $F$ then the process in R.H.S. of \eqref{eqngcpcomp} is a compound version of NGCP $\{\mathcal{M}_0(t)\}_{t\ge0}$. For $k_0=1$, it simplifies to the compound version of non-homogeneous Poisson process. When $F=\delta_1$, the Dirac measure at 1, it simplifies to the NGCP. Furthermore, when $k_0=1$ and $F=\delta_1$, it reduces to the non-homogeneous Poisson process.
%\end{remark}}
Further, its state probabilities  $\hat{q}(n,t)=\mathrm{Pr}\{\hat{\mathcal{M}}(t)=n\}$, $n\ge0$ can be obtained as follows:
\begin{align}
\hat{q}(n,t)&=\sum_{m=0}^{\infty}\mathrm{Pr}\{M(m)=n\}\mathrm{Pr}\{\mathcal{M}_0(t)=m\}\label{Ngaa}\\
%&=\sum_{m=0}^{\infty}\mathrm{Pr}\{M(m)=n\}\sum_{\Omega(k,m)}\prod_{j_0=1}^{k_0}\frac{(\rho_{j_0}(t))^{x_{j_0}}}{x_{j_0}!}e^{-\rho_{j_0}(t)}\\
&=\sum_{\Omega(k,n)}\Big(\prod_{j=1}^{k}\frac{\lambda_j^{n_j}}{n_j!}\Big)\sum_{m=0}^{\infty}m^{z_k}e^{-\lambda m}\sum_{\Omega(k,m)}\prod_{j_0=1}^{k_0}\frac{(\rho_{j_0}(t))^{x_{j_0}}}{x_{j_0}!}e^{-\rho_{j_0}(t)}\nonumber\\
&=\sum_{\Omega(k,n)}\Big(\prod_{j=1}^{k}\frac{\lambda_j^{n_j}}{n_j!}\Big)\sum_{\substack{x_{j_0}\ge0\\1\le j_0\le k_0}}\sum_{\sum_{j_0=1}^{k_0}r_{j_0}=z_{k}}z_{k}!\prod_{j_0=1}^{k_0}\frac{(j_0x_{j_0})^{r_{j_0}}}{r_{j_0}!x_{j_0}!}(e^{-j_0\lambda}\rho_{j_0}(t))^{x_{j_0}} e^{-\rho_{j_0}(t)}\nonumber\\
&=\sum_{\Omega(k,n)}\Big(\prod_{j=1}^{k}\frac{\lambda_j^{n_j}}{n_j!}\Big)\sum_{\sum_{j_0=1}^{k_0}r_{j_0}=z_{k}}z_{k}!\prod_{j_0=1}^{k_0}\frac{j_0^{r_{j_0}}}{r_{j_0}!}e^{-\rho_{j_0}(t)(1-e^{-j_0\lambda})}\mathcal{B}_{r_{j_0}}(\rho_{j_0}(t)e^{-j_0\lambda}),\nonumber
\end{align}
where $z_k=n_1+n_2+\dots+n_k$ and $\rho_{j_0}(t)=\int_{0}^{t}\mu_{j_0}(s)\,\mathrm{d}s$.
In the second step, we have used the pmf of non-homogeneous GCP $\{\mathcal{M}_0(t)\}_{t\ge0}$ (see Kataria \textit{et al.} (2022), Section 3). 

Next, we obtain the system of differential equations that governs  the state probabilities of non-homogeneous IGCP as follows:
 
On differentiating \eqref{Ngaa} and by using Eq. (3.6) of Kataria \textit{et al.} (2022), we get
\begin{align*}
\frac{\mathrm{d}}{\mathrm{d}t}\hat{q}(n,t)&=\sum_{j_0=1}^{k_0}\mu_{j_0}(t) \Big(\sum_{m=0}^{\infty}p(n,m)\mathrm{Pr}\{\mathcal{M}_0(t)=m-j_0\}-\hat{q}(n,t)\Big)\\
&=\sum_{j_0=1}^{k_0}\mu_{j_0}(t)\Big(\sum_{r=0}^{n}p(r,j_0)\sum_{m=j_0}^{\infty}p(n-r,m-j_0)\mathrm{Pr}\{\mathcal{M}_0(t)=m-j_0\}-\hat{q}(n,t)\Big) \\
&=\sum_{j_0=1}^{k_0}\mu_{j_0}(t)\Big(\sum_{r=0}^{n}p(r,j_0)\hat{q}(n-r,t)-\hat{q}(n,t)\Big),\ n\ge0 
\end{align*}
with initial condition $\hat{q}(n,0)=\delta_{n}(0)$.

Its mean and variance are given by
\begin{equation*}
	\mathbb{E}(\hat{\mathcal{M}}(t))=\sum_{j=1}^{k}j\lambda_j\sum_{j_0=1}^{k_0}j_0\rho_{j_0}(t)
	\end{equation*}
and
\begin{equation*}
	\operatorname{Var}(\hat{\mathcal{M}}(t))=\Big(\sum_{j=1}^{k}j\lambda_{j}\Big)^2\sum_{j_0=1}^{k_0}j_0^2\rho_{j_0}(t)+\sum_{j=1}^{k}j^2\lambda_{j}\sum_{j_0=1}^{k_0}j_0\rho_{j_0}(t),
\end{equation*}
respectively.

Let $\tau_n=\inf\{t\ge0:\hat{\mathcal{M}}(t)=n\}$, $n\ge0$. Then, its distribution $F_{\tau_n}(t)=\mathrm{Pr}\{\tau_n\le t\}$ is given by
\begin{align*}
	F_{\tau_n}(t)&=\mathrm{Pr}\{\hat{\mathcal{M}}(t)\ge n\}\\
	&=1-\sum_{m=0}^{n-1}\sum_{\Omega(k,m)}\Big(\prod_{j=1}^{k}\frac{\lambda_j^{m_j}}{m_j!}\Big)\sum_{\sum_{j_0=1}^{k_0}r_{j_0}=z_{k}}z_{k}!\prod_{j_0=1}^{k_0}\frac{(j_0x_{j_0})^{r_{j_0}}}{r_{j_0}!}e^{-\rho_{j_0}(t)(1-e^{-j_0\lambda})}\mathcal{B}_{r_{j_0}}(\rho_{j_0}(t)e^{-j_0\lambda}).
\end{align*}
Observe that $F_{\tau_n}(t)$ is a distribution function if and only if $\rho_{j_0}(t)$'s satisfy the conditions stated in Remark 5 of Leonenko \textit{et al.} (2017).

Moreover, on following along the similar lines to the proof of Proposition \ref{prepmar}, it can be shown that the process $\{\hat{\mathcal{M}}(t)-\sum_{j=1}^{k}j\lambda_j\sum_{j_0=1}^{k_0}j_0\rho_{j_0}(t)\}_{t\ge0}$ is a martingale with respect to natural filtration $\mathcal{F}_t=\sigma(\hat{\mathcal{M}}(s),\, 0<s\le t)$.
%\vspace{3cm}
%\paragraph{Case II} Let us consider the non-homogeneous version of $\{M(t)\}_{t\ge0}$ in \eqref{IGCPrep}. Then, we have the following process:
%\begin{equation*}
%\hat{\mathscr{M}}(t)\coloneqq \mathcal{M}(M_0(t)),\, t\ge0,
%\end{equation*}
%where $\{\mathcal{M}(t)\}_{t\ge0}$ is a NGCP with deterministic and time-dependent intensity functions $\lambda_j(t):[0,\infty)\to [0,\infty)$, $j=1,2,\dots,k$, and it is independent of $\{M_0(t)\}_{t\ge0}$.
%
%For $v\ge0$, the increment process of $\{\hat{\mathscr{M}}(t)\}_{t\ge0}$ is defined as
%\begin{equation*}
%\hat{I}(t,v)\coloneqq \mathcal{M}(M_0(t)+v)-\mathcal{M}(v),\, t\ge0,
%\end{equation*}
%Its marginal distribution $\hat{q}_n(t,v)=\mathrm{Pr}\{\mathcal{M}(M_0(t)+v)-\mathcal{M}(v)=n\}$ can be obtained as follows:
%\begin{align*}
%\hat{q}_n(t,v)&=\sum_{m=0}^{\infty}\mathrm{Pr}\{M_0(t)=m\}\mathrm{Pr}\{\mathcal{M}(m+v)-\mathcal{M}(v)=n\}\\
%&=\sum_{m=0}^{\infty}\sum_{\Omega(k_0,m)}\Big(\prod_{j_0=1}^{k_0}\frac{(\lambda_{j_0}t)^{x_{j_0}}}{x_{j_0}!}e^{-\lambda_{j_0}t}\Big)\sum_{\Omega(k,n)}\prod_{j=1}^{k}\frac{(\Lambda_j(v,m+v))^{n_j}}{n_j!}e^{-\lambda_j(v,m+v)}\\
%&=\sum_{\Omega(k,n)} \sum_{m=0}^{\infty}\sum_{\Omega(k_0,m)}\prod_{j_0=1}^{k_0}\frac{(\lambda_{j_0}t)^{x_{j_0}}}{x_{j_0}!}e^{-\lambda_{j_0}t}\sum_{\sum_{j_0=1}^{k_0}r_{j_0}=z_k}z_{k}!\frac{(e^{-j_0\lambda}\lambda_{j_0}(v,t+v))^{x_{j_0}}}{x_{j_0}}!.
%\end{align*}
%The mean and variance of $\{\hat{\mathscr{M}}(t)\}_{t\ge0}$ are given by
%\begin{align*}
%\mathbb{E}(\hat{\mathscr{M}}(t))&=\sum_{j_0=1}^{k_0}j_0\rho_{j_0}(t)\sum_{j=1}^{k}j\lambda_j 
%\end{align*}

For $v\ge0$, the increment process of  non-homogeneous IGCP is defined as
\begin{equation*}
\hat{I}(t,v)\coloneqq M(\mathcal{M}_0(t+v))-M(\mathcal{M}_0(v))\overset{d}{=}M(I_0(t,v)),\, t\ge0,
\end{equation*}
where $\{I_0(t,v)\}_{t\ge0}$ is the increment process of non-homogeneous GCP $\{\mathcal{M}_0(t)\}_{t\ge0}$ (see Kataria \textit{et al.} (2022)).
Its marginal distribution $\hat{q}_n(t,v)=\mathrm{Pr}\{\hat{I}(t,v)=n\}$, $n\ge0$ can be obtained as follows:
\begin{align}
\hat{q}_n(t,v)&=\sum_{m=0}^{\infty}\mathrm{Pr}\{I_0(t,v)=m\}\mathrm{Pr}\{M(m)=n\}\label{NIgcpon}\\
&=\sum_{\Omega(k,n)}\Big(\prod_{j=1}^{k}\frac{\lambda_j^{n_j}}{n_j!}\Big)\sum_{m=0}^{\infty}m^{z_k}e^{-m\lambda}\sum_{\Omega(k_0,m)}\Big(\prod_{j_0=1}^{k_0}\frac{(\rho_{j_0}(v,t+v))^{x_{j_0}}}{x_{j_0}!}e^{-\rho_{j_0}(v,t+v)}\Big)\nonumber\\
&=\sum_{\Omega(k,n)}\Big(\prod_{j=1}^{k}\frac{\lambda_j^{n_j}}{n_j!}\Big)\sum_{\sum_{j_0=1}^{k_0}r_{j_0}=z_k}z_{k}!\prod_{j_0=1}^{k_0}\frac{j_0^{r_{j_0}}}{r_{j_0}!}e^{-\rho_{j_0}(v,t+v)(1-e^{-j_0\lambda})}\mathcal{B}_{r_{j_0}}(e^{-j_0\lambda}\rho_{j_0}(v,t+v)),\label{lk}
\end{align}
where $\rho_{j_0}(v,t+v)=\int_{v}^{t+v}\mu_{j_0}(s)\,\mathrm{d}s$ 
and $z_k=n_1+n_2+\dots+n_k$.
\begin{remark}
On substituting $v=0$ in \eqref{lk}, we obtain the state probabilities of IGCP given in \eqref{thmpmf}. 
\end{remark}

On differentiating \eqref{NIgcpon} and by using Remark 3.5 of Kataria \textit{et al.} (2022), we get
\begin{align*}
\frac{\mathrm{d}}{\mathrm{d}t}\hat{q}_n(t,v)&=\sum_{m=0}^{\infty}p(n,m)\sum_{j_0=1}^{k_0}\mu_{j_0}(t+v)\Big(\mathrm{Pr}\{I_0(t,v)=m-j_0\}-\mathrm{Pr}\{I_0(t,v)=m\}\Big)\\
&=\sum_{j_0=1}^{k_0}\mu_{j_0}(t+v)\Big(\sum_{m=j_0}^{\infty}\sum_{r=0}^{n}p(r,j_0)p(n-r,m-j_0)\mathrm{Pr}\{I_0(t,v)=m-j_0\}-\hat{q}_n(t,v)\Big)\\
&=\sum_{j_0=1}^{k_0}\mu_{j_0}(t+v)\Big(\sum_{r=0}^{n}p(r,j_0)\hat{q}_{n-r}(t,v)-\hat{q}_n(t,v)\Big),\ n\ge0
\end{align*}
with initial condition $\hat{q}_n(0,v)=\delta_n(0)$.

\section{Some extensions of IGCP}
Here, we study few extended versions of the IGCP, namely, the compound IGCP, the multivariate IGCP and the $q$-iterated GCP. First, we consider a compound version of it.
\subsection{Compound iterated generalized counting process} Khandakar and Kataria (2024) introduced and studied a compound version of the GCP, namely, the compound generalized counting process (CGCP). We denote it by $\{W(t)\}_{t\ge0}$,. It is defined as follows (see Khandakar and Kataria (2024)):
\begin{equation*}
	W(t)\coloneqq\sum_{i=1}^{M(t)}X_i,\ t\ge0,
\end{equation*}
where $X_i$ are iid random variables independent of the GCP.

Here, we define the compound IGCP $\{Z(t)\}_{t\ge0}$ as a time-changed variant of the CGCP by subordinating  it with a GCP $\{M_0(t)\}_{t\ge0}$. That is,  
\begin{equation}\label{compigrep}
Z(t)\coloneqq W(M_0(t))=\sum_{i=1}^{M(M_0(t))}X_i,\, t\ge0,
\end{equation}
where $\{W(t)\}_{t\ge0}$ is  independent $\{M_0(t)\}_{t\ge0}$.    

 The pgf of CGCP can be obtained as follows:
\begin{align}
G_W(u,t)&=\mathbb{E}(\mathbb{E}(u^{W(t)}|M(t)))\nonumber\\
&=\sum_{m=0}^{\infty}(\mathbb{E}(u^{X_1}))^m\mathrm{Pr}\{M(t)=m\}, \, \text{(as $X_i$'s are iid random variables)}\nonumber\\
&=\exp\Big(-\sum_{j=1}^{k}\lambda_{j}t(1-(\mathbb{E}(u^{X_1}))^{j})\Big),\label{pgfcgcp}
\end{align}
where in the last step we have used \eqref{pgfmt}.
Its distribution function is given by
\begin{align}
H_W(w,t)&=\sum_{m=0}^{\infty}\mathrm{Pr}\{M(t)=m\}H_{X_1}^{*(m)}(w)\nonumber\\
&=\mathbb{I}_{\{w\ge0\}}e^{-\lambda t}+\sum_{m=1}^{\infty}\mathrm{Pr}\{M(t)=m\}H_{X_1}^{*(m)}(w),\label{cdfyt}
\end{align}
where $H_{X_1}^{*(m)}(\cdot)$ is the $m$-fold convolution of the distribution function $H_{X_1}(\cdot)$ of $X_1$.
\begin{remark}
For $k=k_0=1$, the  compound IGCP reduces to the compound iterated Poisson process. If $X_1$ has an atom at 1, that is, $H_{X_1}(x)=\mathbb{I}_{\{x\ge1\}}$ then the compound IGCP reduces to the IGCP and the compound iterated Poisson process reduces to the iterated Poisson process.
\end{remark}
For the subsequent results, we recall $\lambda$, $\mu$, $z_{k}$ and $\Omega(k,m)$ from Section \ref{secigcp}.
\begin{proposition}
The distribution function  $\hat{H}(w,t)=\mathrm{Pr}\{Z(t)\le w\}$ is given by
{\small\begin{align*}
\hat{H}(w,t)&=\mathbb{I}_{\{w\ge0\}}\exp\Big(-\sum_{j_0=1}^{k_0}\mu_{j_0}t(1-e^{-j_0\lambda})\Big)\\
&\ \  +\sum_{m=1}^{\infty}H_X^{*(m)}(w)\sum_{\Omega(k,m)}\Big(\prod_{j=1}^{k}\frac{\lambda_{j}^{m_{j}}}{m_{j}!}\Big)e^{-\mu t}\sum_{\sum_{j_0=1}^{k_0}r_{j_0}=z_{k}}z_{k}!\prod_{j_0=1}^{k_0}\frac{j_0^{r_{j_0}}}{r_{j_0}!}e^{\mu_{j_0}te^{-j_0\lambda}}\mathcal{B}_{r_{j_0}}(e^{-j_0\lambda}\mu_{j_0}t),
\end{align*}}
where $\mathcal{B}_n(x)$ is the $n$th order Bell polynomial defined in \eqref{bell}.
\end{proposition}
\begin{proof}
By using \eqref{compigrep} and \eqref{cdfyt}, we get
{\small\begin{align*}
\hat{H}(w,t)&=\sum_{n=0}^{\infty}\mathrm{Pr}\{M_0(t)=n\}H_W(w,n)\\
&=\mathbb{I}_{\{w\ge0\}}\sum_{n=0}^{\infty}\mathrm{Pr}\{M_0(t)=n\}e^{-\lambda n}+\sum_{n=0}^{\infty}\mathrm{Pr}\{M_0(t)=n\}\sum_{m=1}^{\infty}\mathrm{Pr}\{M(n)=m\}H_X^{*(m)}(w)\\
&=\mathbb{I}_{\{w\ge0\}}\exp\Big(-\sum_{j_0=1}^{k_0}\mu_{j_0}t(1-e^{-j_0\lambda})\Big)+\sum_{m=1}^{\infty}H_X^{*(m)}(w)\sum_{n=0}^{\infty}\mathrm{Pr}\{M_0(t)=n\}\mathrm{Pr}\{M(n)=m\},
%&=\mathbb{I}_{\{w\ge0\}}\exp\Big(-\sum_{j_0=1}^{k_0}\mu_{j_0}t(1-e^{-j_0\lambda})\Big)+\sum_{m=1}^{\infty}H_X^{*(m)}(w)\sum_{\Omega(k,m)}\Big(\prod_{j=1}^{k}\frac{\lambda_{j}^{m_{j}}}{m_{j}!}\Big)e^{-\mu t}\\
%&\hspace{4cm} \cdot\sum_{\sum_{j_0=1}^{k_0}r_{j_0}=z_{k}}z_{k}!\prod_{j_0=1}^{k_0}\frac{j_0^{r_{j_0}}}{r_{j_0}!}e^{\mu_{j_0}te^{-j_0\lambda}}\mathcal{B}_{r_{j_0}}(e^{-j_0\lambda}\mu_{j_0}t),
\end{align*}}
where we have used \eqref{pgfmt} in the last step. Finally, the result follows by using \eqref{IGpmf}.
\end{proof}
\begin{remark}
If $X_i$'s are absolutely continuous random variables with probability density function (pdf) $h_X(\cdot)$ then the absolute continuous component of the pdf of $\{Z(t)\}_{t\ge0}$ is given by
{\small\begin{equation*}
\hat{h}(w,t)=\sum_{m=1}^{\infty}h_X^{*(m)}(w)\sum_{\Omega(k,m)}\Big(\prod_{j=1}^{k}\frac{\lambda_{j}^{m_{j}}}{m_{j}!}\Big)e^{-\mu t}\sum_{\sum_{j_0=1}^{k_0}r_{j_0}=z_{k}}z_{k}!\prod_{j_0=1}^{k_0}\frac{j_0^{r_{j_0}}}{r_{j_0}!}e^{\mu_{j_0}te^{-j_0\lambda}}\mathcal{B}_{r_{j_0}}(e^{-j_0\lambda}\mu_{j_0}t)
\end{equation*}}
and its discrete component is given by
\begin{align*}
\mathrm{Pr}\{Z(t)=0\}&=\sum_{m=0}^{\infty}\mathrm{Pr}\{M_0(t)=m\}\mathrm{Pr}\{W(m)=0\}=\exp\Big(-\sum_{j_0=1}^{k_0}\mu_{j_0}t(1-e^{-j_0\lambda})\Big).
\end{align*}

\end{remark}
\begin{remark}
If $X_i$'s are discrete random variables then the pmf of $\{Z(t)\}_{t\ge0}$ is given by
{\footnotesize\begin{align}
\mathrm{Pr}\{Z(t)=n\}&=\mathbb{I}_{\{n=0\}}\exp\Big(-\sum_{j_0=1}^{k_0}\mu_{j_0}t(1-e^{-j_0\lambda})\Big)\nonumber\\
&\ \  +\sum_{m=1}^{\infty}\Psi_X^{*(m)}(n)\sum_{\Omega(k,m)}\Big(\prod_{j=1}^{k}\frac{\lambda_{j}^{m_{j}}}{m_{j}!}\Big)e^{-\mu t}\sum_{\sum_{j_0=1}^{k_0}r_{j_0}=z_{k}}z_{k}!\prod_{j_0=1}^{k_0}\frac{j_0^{r_{j_0}}}{r_{j_0}!}e^{\mu_{j_0}te^{-j_0\lambda}}\mathcal{B}_{r_{j_0}}(e^{-j_0\lambda}\mu_{j_0}t),\label{mnmnm}
\end{align}}
where $\Psi_X^{*(m)}(n)=\mathrm{Pr}\{X_1+X_2+\dots+X_m=n\}$, $n\in\mathbb{Z}$.
\end{remark}
\begin{example}
Let $\{M(t)\}_{t\ge0}$ be a GCP with positive rates $\beta_1$, $\beta_2$, $\dots$, $\beta_r$ such that $X_1\overset{d}{=}M(1)$. Then, 
\begin{equation}\label{disgcconv}
\Psi_X^{*(m)}(n)=\mathrm{Pr}\{X_1+X_2+\dots+X_m=n\}=\sum_{\Theta(r,n)}\prod_{j=1}^{r}\frac{(m\beta_j)^{x_j}}{x_j!}e^{-m\beta_j},\, n\ge0,
\end{equation}
where $\Theta(r,n)=\{(x_1,x_2,\dots,x_r):\sum_{j=1}^{r}jx_j=n,\, x_j\in\mathbb{N}_0\}$.
On substituting \eqref{disgcconv} in \eqref{mnmnm}, we get
\begin{align*}
\mathrm{Pr}\{Z(t)=n\}&=	\mathbb{I}_{\{n=0\}}\exp\Big(-\sum_{j_0=1}^{k_0}\mu_{j_0}t(1-e^{-j_0\lambda})\Big)+\sum_{m=1}^{\infty}\sum_{\Theta(r,n)}\prod_{j=1}^{r}\Big(\frac{(m\beta_j)^{x_j}}{x_j!}e^{-m\beta_j}\Big)\\
	&\ \  \cdot\sum_{\Omega(k,m)}\Big(\prod_{j=1}^{k}\frac{\lambda_{j}^{m_{j}}}{m_{j}!}\Big)e^{-\mu t}\sum_{\sum_{j_0=1}^{k_0}r_{j_0}=z_{k}}z_{k}!\prod_{j_0=1}^{k_0}\frac{j_0^{r_{j_0}}}{r_{j_0}!}e^{\mu_{j_0}te^{-j_0\lambda}}\mathcal{B}_{r_{j_0}}(e^{-j_0\lambda}\mu_{j_0}t).
\end{align*}
For $n=0$, we have
\begin{align*}
\mathrm{Pr}\{Z(t)=0\}&=	\exp\Big(-\sum_{j_0=1}^{k_0}\mu_{j_0}t(1-e^{-j_0\lambda})\Big)+\sum_{m=1}^{\infty}e^{-m\beta}\sum_{\Omega(k,m)}\Big(\prod_{j=1}^{k}\frac{\lambda_{j}^{m_{j}}}{m_{j}!}\Big)e^{-\mu t}\\
&\hspace{4cm} \cdot\sum_{\sum_{j_0=1}^{k_0}r_{j_0}=z_{k}}z_{k}!\prod_{j_0=1}^{k_0}\frac{j_0^{r_{j_0}}}{r_{j_0}!}e^{\mu_{j_0}te^{-j_0\lambda}}\mathcal{B}_{r_{j_0}}(e^{-j_0\lambda}\mu_{j_0}t),
\end{align*}
where $\beta=\beta_1+\beta_2+\dots+\beta_r$.
\end{example}
\begin{example}
Let $X_1$ has geometric distribution with parameter $p\in(0,1]$. Then,
\begin{equation*}
\Psi_{X_1}^{*(m)}=\binom{n-1}{m-1}p^m(1-p)^{n-m},\  n\ge m.
\end{equation*}
Thus,
\begin{align*}
\mathrm{Pr}\{Z(t)=n\}&=	\mathbb{I}_{\{n=0\}}\exp\Big(-\sum_{j_0=1}^{k_0}\mu_{j_0}t(1-e^{-j_0\lambda})\Big)+\sum_{m=1}^{\infty}\binom{n-1}{m-1}p^m(1-p)^{n-m}\\
	&\ \  \cdot\sum_{\Omega(k,m)}\Big(\prod_{j=1}^{k}\frac{\lambda_{j}^{m_{j}}}{m_{j}!}\Big)e^{-\mu t}\sum_{\sum_{j_0=1}^{k_0}r_{j_0}=z_{k}}z_{k}!\prod_{j_0=1}^{k_0}\frac{j_0^{r_{j_0}}}{r_{j_0}!}e^{\mu_{j_0}te^{-j_0\lambda}}\mathcal{B}_{r_{j_0}}(e^{-j_0\lambda}\mu_{j_0}t).
\end{align*}
\end{example}
Next, we obtain the finite dimensional distribution of compound IGCP.
\begin{theorem}
Let $0=t_0\le t_1\le\dots\le t_n=t$ be a partition of $[0,t]$ and $h_X^{*(m)}(\cdot)$ denotes the $m$-fold convolution of the density function $h_{X_1}(\cdot)$ of $X_1$. Then, the finite dimensional distribution of compound IGCP has the following form:
{\small\begin{align}\label{conv}
\hat{H}_{Z(t_1), Z(t_2),\dots, Z(t_n)}(x_1,x_2,\dots,x_n)&=\sum_{\substack{m_l\ge0\\ l=1,2,\dots,n}}\Big(\prod_{l=1}^{n}\hat{p}(m_l,\Delta t_l)\Big)\nonumber\\
&\ \  \cdot\int_{-\infty}^{u_1}\int_{-\infty}^{u_2}\dots \int_{-\infty}^{u_n}\Big(\prod_{r=1}^{n}h_{X_{1}}^{*(m_r)}(y_r)\Big)\mathrm{d}y_n\mathrm{d}y_{n-1}\dots\mathrm{d}y_1,
\end{align}}
where $\Delta t_l=t_l-t_{l-1}$ and $u_r=x_r-(y_1+y_2+\dots+y_{r-1})$.
\end{theorem}
\begin{proof}
As $\{\hat{M}(t)\}_{t\ge0}$ has independent and stationary increments, we have $\hat{M}(t_r)=$ $\hat{M}(\Delta t_1)+\hat{M}(\Delta t_2)+\dots+\hat{M}(\Delta t_r)$, $r=1,2,\dots,n$. Let 
\begin{align*}
X(m_1)&=X_1+X_2+\dots+X_{m_{1}},\\
X(m_2)&=X_{m_{1}+1}+X_{m_{1}+2}+\dots+X_{m_{1}+m_{2}},\\
&\ \vdots\\
X(m_n)&=X_{m_{1}+m_{2}+\dots+m_{n-1}+1}+X_{m_{1}+m_{2}+\dots+m_{n-1}+2}+\dots+X_{m_{1}+m_{2}+\dots+m_{n}}.
\end{align*}
Then, 
\begin{align*}
\hat{H}&_{Z(t_1), Z(t_2),\dots, Z(t_n)}(x_1,x_2,\dots,x_n)\\
&=\mathrm{Pr}\{Z(t_1)\le x_1, Z(t_2)\le x_2,\dots, Z(t_n)\le x_n\}\\ &=\sum_{\substack{m_l\ge0\\l=1,2,\dots,n}}\mathrm{Pr}\Big\{\sum_{i=1}^{m_1}X_i\le x_1, \sum_{i=1}^{m_1+m_2}X_i\le x_2,\dots, \sum_{i=1}^{\sum_{l=1}^{n}m_l}X_i\le x_n\Big\}\prod_{l=1}^{n}\hat{p}(m_l,\Delta t_l)\\
&=\sum_{\substack{m_l\ge0\\l=1,2,\dots,n}}\mathrm{Pr}\Big\{X(m_1)\le x_1, X(m_1)+X(m_2)\le x_2, \dots, \sum_{l=1}^{n}X(m_l)\le x_n\Big\}\prod_{l=1}^{n}\hat{p}(m_l,\Delta t_l)\\
&=\sum_{\substack{m_l\ge0\\l=1,2,\dots,n}}\Big(\prod_{l=1}^{n}\hat{p}(m_l,\Delta t_l)\Big)\int_{-\infty}^{x_1}\int_{-\infty}^{x_2 -y_1}\dots \int_{-\infty}^{x_n-\sum_{l=1}^{n-1}y_l}\Big(\prod_{r=1}^{n}h_{X_{1}}^{*(m_r)}(y_r)\Big)\mathrm{d}y_n\mathrm{d}y_{n-1}\dots\mathrm{d}y_1.
\end{align*}
Finally, the change of variables gives the required result.
\end{proof}
\begin{remark}
For $n=1$, \eqref{conv} reduces to the marginal distribution $\hat{H}_{Z(t)}(x)=\mathrm{Pr}\{Z(t)\le x\}$ of compound IGCP, that is,
\begin{equation*}
\hat{H}_{Z(t)}(x)=\sum_{m\ge0}\hat{p}(m,t)\int_{-\infty}^{x}h_{X_1}^{*(m)}(y)\mathrm{d}y.
\end{equation*}
\end{remark}
\begin{example}
If $X_1$ has exponential distribution with parameter $\lambda>0$ then the marginal distribution of compound IGCP is given by
\begin{equation*}
\hat{H}_{Z(t)}(x)=\sum_{m\ge0}\hat{p}(m,t)\int_{-\infty}^{x}h_{X_1}^{*(m)}(y)\mathrm{d}y=\sum_{m\ge0}\frac{\gamma(m,\lambda x)}{(m-1)!}\hat{p}(m,t), 
\end{equation*}
where $\gamma(s,t)=\int_{0}^{t}e^{-x}x^{s-1}\mathrm{d}x$, $t\ge0$ is the incomplete gamma function.
\end{example}

By using \eqref{compigrep} and \eqref{pgfcgcp}, the pgf $\hat{G}_{Z(t)}(u)=\mathbb{E}(u^{Z(t)})$ can be obtained as follows:
\begin{align*}
\hat{G}_{Z(t)}(u)&=\mathbb{E}(\mathbb{E}(u^{W(M_0(t))}|M_0(t)))\\
&=\sum_{m=0}^{\infty}\mathrm{Pr}\{M_0(t)=m\}\mathbb{E}(u^{W(m)})\\
&=\sum_{m=0}^{\infty}\mathrm{Pr}\{M_0(t)=m\}\exp\Big(-\sum_{j=1}^{k}\lambda_{j}m\Big(1-(\mathbb{E}(u^{X_1}))^{j}\Big)\Big)\\
&=\exp\Big(-\sum_{j_0=1}^{k_0}\mu_{j_0}t\Big(1-\exp\Big(-j_0\sum_{j=1}^{k}\lambda_{j}(1-(\mathbb{E}(u^{X_1}))^{j})\Big)\Big)\Big), \ |u|\le1.
\end{align*}
Thus, it satisfies
\begin{equation*}
\frac{\partial}{\partial t}\hat{G}_{Z(t)}(u)=-\sum_{j_0=1}^{k_0}\mu_{j_0}\Big(1-\exp\Big(-j_0\sum_{j=1}^{k}\lambda_{j}(1-(\mathbb{E}(u^{X_1}))^{j})\Big)\Big)\hat{G}_{Z(t)}(u),\ \  \hat{G}_{Z(0)}(u)=1.
\end{equation*}

By using Theorem 2.1 of Leonenko \textit{et al.} (2014), the mean and variance of compound IGCP are given by
\begin{equation*}
\mathbb{E}(Z(t))=\mathbb{E}(W(1))\mathbb{E}(M_0(t))=\sum_{j=1}^{k}j\lambda_{j}\sum_{j_0=1}^{k_0}j_0\mu_{j_0}t\mathbb{E}(X_1)
\end{equation*}
and
\begin{align*}
\operatorname{Var}(Z(t))&=(\mathbb{E}(W(1)))^2\operatorname{Var}(M_0(t))+\operatorname{Var}(W(1))\mathbb{E}(M_0(t))\\
&=\Big(\sum_{j=1}^{k}j\lambda_{j}\mathbb{E}(X_1)\Big)^2\sum_{j_0=1}^{k_0}j_0^2\mu_{j_0}t+\sum_{j_0=1}^{k_0}j_0\mu_{j_0}t\Big(\operatorname{Var}(X_1)\sum_{j=1}^{k}j\lambda_{j}+(\mathbb{E}(X_1))^2\sum_{j=1}^{k}j^2\lambda_{j}\Big),
\end{align*}
respectively. Here, we have used the mean and variance of CGCP (see Khandakar and Kataria (2024), Section 4).
\begin{remark}
The following limiting result holds true:
\begin{equation*}
\lim_{t\to\infty}\frac{Z(t)}{t}=\sum_{j=1}^{k}j\lambda_{j}\sum_{j_0=1}^{k_0}j_0\mu_{j_0}t\mathbb{E}(X_1)\ \  \text{with probability $1$}
\end{equation*}
which follows by using the strong law of large numbers.
\end{remark}
\begin{proposition}\label{prpcgpart}
Let $X_i$'s be iid random variables such that $\mathrm{Pr}\{X_1=i\}=\alpha_i$, $i=0$, $1$, $2$, $\dots$ and  $D(t)=\sum_{i=1}^{\hat{M}(t)}X_i$. Then, the pgf $\hat{G}_{D(t)}(u)=\mathbb{E}(u^{D(t)})$ is given by
\begin{equation*}
\hat{G}_{D(t)}(u)=\exp\Big(-t\sum_{j_0=1}^{k_0}\mu_{j_0}\Big(1-\exp\Big(-j_0\sum_{j=1}^{k}\lambda_j\sum_{i=1}^{\infty}\alpha_i^{*(j)}(1-u^i)\Big)\Big)\Big), \ \ |u|\le1,
\end{equation*}
where $\alpha_i^{*(j)}=\displaystyle\sum_{\substack{\sum_{m=1}^{j}r_m=i\\r_m\in\mathbb{N}_0}}\alpha_{r_1}\alpha_{r_2}\dots\alpha_{r_j}$.
\end{proposition}
\begin{proof}
See Appendix A3.
\end{proof}

\begin{remark}\label{remcg}
Let $\{Z_i\}_{i\ge1}$ be a sequence of iid random variables such that $\mathrm{Pr}\{Z_1=i\}=\frac{1}{\lambda}\sum_{j=1}^{k}\lambda_j\alpha_i^{*(j)}$. Also, let $\{M_0(t)\}_{t\ge0}$ be a GCP with positive rates $\mu_1$, $\mu_2$, $\dots$, $\mu_{k_0}$ and $\{Y_i\}_{i\ge1}$ be a sequence of iid random variables such that $Y_1\overset{d}{=}\sum_{i=1}^{N(1)}Z_i$, where $\{N(t)\}_{t\ge0}$ is a Poisson process with parameter $\lambda$ and independent of $Z_i$'s. Then, we have
\begin{equation*}
D(t)\overset{d}{=}\sum_{i=1}^{M_0(t)}Y_i,
\end{equation*}
where $Y_i$'s are independent of the GCP $\{M_0(t)\}_{t\ge0}$.
\end{remark}
Next, we give a martingale characterization of the compound IGCP $\{D(t)\}_{t\ge0}$ defined in Proposition \ref{prpcgpart}.
%\begin{theorem}
%The process $\{D(t)\}_{t\ge0}$ is a compound IGCP iff $\{D(t)-\sum_{j=1}^{k}j\lambda_j\sum_{j_0=1}^{k_0}j_0\mu_{j_0}t$ $\mathbb{E}(X_1)\}_{t\ge0}$ is a $\{\mathcal{F}_t\}_{t\ge0}$-martingale, where $\mathcal{F}_t=\sigma(D(s),\ s\le t)$.
%\end{theorem}
\begin{theorem}
	Let $\{X_i\}_{i\ge1}$ be a sequence of iid random variables with non-negative integer support. Then, the process $D(t)=\sum_{i=1}^{\hat{M}(t)}X_i$, $t\ge0$, where $X_i$'s are independent of $\{\hat{M}(t)\}_{t\ge0}$ is a compound IGCP iff $\{D(t)-\sum_{j=1}^{k}j\lambda_j$ $\sum_{j_0=1}^{k_0}j_0\mu_{j_0}t$ $\mathbb{E}(X_1)\}_{t\ge0}$ is a $\{\mathcal{F}_t\}_{t\ge0}$-martingale, where $\mathcal{F}_t=\sigma(D(s),\ s\le t)$.
\end{theorem}
\begin{proof}
From Remark \ref{remcg}, the compound IGCP $\{D(t)\}_{t\ge0}$ is equal in distribution to a compound GCP. So, the proof follows on using Remark 4.2 of Khandakar and Kataria (2024) and Theorem 5.2 of Zhang and Li (2016).
\end{proof}

\subsection{Multivariate IGCP}
Let $\{M_1(t)\}_{t\ge0}$, $\{M_2(t)\}_{t\ge0}$, $\dots$, $\{M_q(t)\}_{t\ge0}$ be $q$ independent GCPs such that $\{M_i(t)\}_{t\ge0}$ performs $k_i$ kinds of jumps of size $j_i$ with positive rate $\lambda_{i{j_i}}$, $j_i=1,2,\dots,k_i$. Kataria and Dhillon (2024) defined a multivariate version of the GCP as follows:
\begin{equation*}
	\bar{M}(t)\coloneqq(M_1(t),M_2(t), \dots, M_q(t)), \, t\ge0.
\end{equation*}
It is called the multivariate GCP.

 Here, we study a multivariate version of the IGCP which is a time-changed variant of the multivariate GCP. It is defined as follows:
\begin{equation*}
\bar{\mathcal{M}}(t)\coloneqq(M_1(M_0(t)),M_2(M_0(t)), \dots, M_q(M_0(t))), \, t\ge0,
\end{equation*}
where $\{M_0(t)\}_{t\ge0}$ is a GCP with positive rates $\mu_{j_0}$, $j_0=1,2,\dots,k_0$ and it is independent of $\{M_i(t)\}_{t\ge0}$.  We call it the multivariate IGCP.

The following notations will be used:
Let $\bar{0}=(0,0,\dots,0)$, $\bar{1}=(1,1,\dots,1)$ and $\bar{n}=(n_1,n_2,\dots,n_q)$ be $q$-tuple vectors. For all $i=1,2,\dots,q$, if  $n_i\ge m_i$ then we denote it by $\bar{n}\ge\bar{m}$. Also, $\bar{n}\succ \bar{m}$ ( $\bar{n}\prec\bar{m}$) stand for $n_i\ge m_i$ ($n_i\leq m_i$) for all $i=1,2,\dots,q$ and $\bar{n}\neq \bar{m}$. 

Let $\Lambda=\sum_{i=1}^{q}\sum_{j_i=1}^{k_i}\lambda_{ij_i}$ and $\Omega(k_i,m_i)=\{(x_{i1},x_{i2}$, $\dots$, $x_{ik_i}):\sum_{j_i=1}^{k_i}j_ix_{ij_i}=m_i,\, x_{ij_i}\in\mathbb{N}_0\}$.

In an infinitesimal time interval of length $h$ such that $o(h)/h\to0$ as $h\to0$, the transition probabilities of multivariate IGCP are given by
{\footnotesize\begin{equation*}
\mathrm{Pr}\{\bar{\mathcal{M}}(t+h)=\bar{n}+\bar{m}|\bar{\mathcal{M}}(t)=\bar{n}\}=\begin{cases}
1-h\sum_{j_0=1}^{k_0}\mu_{j_0}(1-e^{-j_0\Lambda})+o(h),\, \bar{m}=\bar{0},\vspace{.2cm}\\
h\sum_{j_0=1}^{k_0}\mu_{j_0}e^{-j_0\Lambda}\displaystyle\sum_{\substack{\Omega(k_i,m_i)\\i=1,2,\dots,q}}\prod_{i=1}^{q}\prod_{j_i=1}^{k_i}\frac{(j_0\lambda_{ij_i})^{x_{ij_i}}}{x_{ij_i}!}+o(h),\, \bar{m}\succ\bar{0}.
\end{cases}
\end{equation*}}
\begin{remark}
For $k_0=k_1=\dots=k_q=1$, the multivariate IGCP reduces to the multivariate version of iterated Poisson process.
\end{remark}
Let $p_0(\cdot, t)$ be the pmf of $\{M_0(t)\}_{t\ge0}$ and $G(\bar{u},t)$ be the pgf of multivarite GCP. Then, for $|u_i|\le1$, the pgf $\hat{G}_{\bar{\mathcal{M}}}(\bar{u},t)=\mathbb{E}(u_1^{M_1(M_0(t))}u_2^{M_2(M_0(t))}$ $\dots$ $u_q^{M_q(M_0(t))})$ of multivariate IGCP can be obtained as follows:
\begin{align*}
\hat{G}_{\bar{\mathcal{M}}}(\bar{u},t)&=\sum_{m=0}^{\infty}G(\bar{u},m)p_0(m,t)\\
&=\sum_{m=0}^{\infty}\exp\Big(-m\sum_{i=1}^{q}\sum_{j_i=1}^{k_i}\lambda_{ij_i}(1-u_i^{j_i})\Big)p_0(m,t)\\
&=\exp\Big(-\sum_{j_0=1}^{k_0}\mu_{j_0}t\Big(1-\exp(-j_0\sum_{i=1}^{q}\sum_{j_i=1}^{k_i}\lambda_{ij_i}(1-u_i^{j_i})\Big)\Big),
\end{align*}
where the penultimate step follows by using Eq. (3.5) of Kataria and Dhillon (2024), and the last step follows by using \eqref{mgfmt}. Thus, 
{\small\begin{equation}\label{MIGCPpgfde}
\frac{\partial}{\partial t}\hat{G}_{\bar{\mathcal{M}}}(\bar{u},t)=-\Big(\sum_{j_0=1}^{k_0}\mu_{j_0}\Big(1-\exp(-j_0\sum_{i=1}^{q}\sum_{j_i=1}^{k_i}\lambda_{ij_i}(1-u_i^{j_i})\Big)\hat{G}_{\bar{\mathcal{M}}}(\bar{u},t),\ \ \hat{G}_{\bar{\mathcal{M}}}(\bar{u},0)=1.
\end{equation}}

\begin{proposition}\label{thmappen}
The pmf $\hat{p}_{\bar{\mathcal{M}}}(\bar{n},t)$, $\bar{n}\ge\bar{0}$ of $\{\bar{\mathcal{M}}(t)\}_{t\ge0}$ solves the following system of differential equations:
{\small\begin{equation*}
\frac{\mathrm{d}}{\mathrm{d}t}\hat{p}_{\bar{\mathcal{M}}}(\bar{n},t)=-\sum_{j_0=1}^{k_0}\mu_{j_0}(1-e^{-j_0\lambda})\hat{p}_{\bar{\mathcal{M}}}(\bar{n},t)+\sum_{\bar{m}\succ\bar{0}}\hat{p}_{\bar{\mathcal{M}}}(\bar{n}-\bar{m},t)\sum_{j_0=1}^{k_0}\mu_{j_0}e^{-j_0\Lambda}\sum_{\substack{\Omega(k_i,m_i)\\i=1,2,\dots,q}}\prod_{i=1}^{q}\prod_{j_i=1}^{k_i}\frac{(j_0\lambda_{ij_i})^{x_{ij_i}}}{x_{ij_i}!}
\end{equation*}}
with initial condition $\hat{p}_{\bar{\mathcal{M}}}(\bar{n},0)=\delta_{\bar{n}}(\bar{0})$. 
\end{proposition}
\begin{proof}
The proof follows similar lines to that of Proposition \ref{prpde}. Thus, it is omitted.
For an alternate proof, we refer the reader to Appendix A2.   

\end{proof}
\begin{theorem}
The pmf of multivariate IGCP is given by
\begin{equation*}
\hat{p}_{\bar{\mathcal{M}}}(\bar{n},t)=\sum_{\substack{\Omega(k_i,n_i)\\i=1,2,\dots,q}}\Big(\prod_{i=1}^{q}\prod_{j_i=1}^{k_i}\frac{(\lambda_{ij_i})^{n_{ij_i}}}{n_{ij_i}!}\Big)\sum_{m=0}^{\infty}m^{\sum_{i=1}^{q}\sum_{j_i=1}^{k_i}n_{ij_i}}\sum_{\Omega(k_0,m)}\prod_{j_0=1}^{k_0}\frac{(\mu_{j_0}e^{-j_0\Lambda}t)^{x_{j_0}}}{x_{j_0}!}e^{-\mu_{j_0}t},\, \bar{n}\ge\bar{0},
\end{equation*}
where $\Omega(k_i,n_i)=\{(n_{i1},n_{i2},\dots,n_{ik_i}):\sum_{j_i=1}^{k_i}j_in_{ij_i}=n_i,\,n_{ij_i}\in\mathbb{N}_0\}$.
\end{theorem}
\begin{proof}
Let $p(\bar{n},t)$ be the pmf of multivariate GCP. By using Eq. (3.7) of Kataria and Dhillon (2024), we get
\begin{align*}
\hat{p}_{\bar{\mathcal{M}}}(\bar{n},t)&=\sum_{m=0}^{\infty}p(\bar{n},m)\mathrm{Pr}\{M_0(t)=m\}\\
&=\sum_{\substack{\Omega(k_i,n_i)\\i=1,2,\dots,q}}\Big(\prod_{i=1}^{q}\prod_{j_i=1}^{k_i}\frac{(\lambda_{ij_i})^{n_{ij_i}}}{n_{ij_i}!}\Big)\sum_{m=0}^{\infty}m^{\sum_{i=1}^{q}\sum_{j_i=1}^{k_i}n_{ij_i}}e^{-m\Lambda}\mathrm{Pr}\{M_0(t)=m\}\\
&=\sum_{\substack{\Omega(k_i,n_i)\\i=1,2,\dots,q}}\sum_{m=0}^{\infty}\Big(\prod_{i=1}^{q}\prod_{j_i=1}^{k_i}\frac{(m\lambda_{ij_i})^{n_{ij_i}}}{n_{ij_i}!}\Big)\sum_{\Omega(k_0,m)}\prod_{j_0=1}^{k_0}\frac{(\mu_{j_0}e^{-j_0\Lambda}t)^{x_{j_0}}}{x_{j_0}!}e^{-\mu_{j_0}t},
\end{align*}
where the last step follows by using \eqref{p(n,t)}. This completes the proof.
\end{proof}
\begin{corollary}
The pmf of multivariate IGCP has following equivalent form:
\begin{equation*}
\hat{p}_{\bar{\mathcal{M}}}(\bar{n},t)=\sum_{\substack{\Omega(k_i,n_i)\\i=1,2,\dots,q}}\Big(\prod_{i=1}^{q}\prod_{j_i=1}^{k_i}\frac{(\lambda_{ij_i})^{n_{ij_i}}}{n_{ij_i}!}\Big)\sum_{\sum_{j_{0}=1}^{k_0}r_{j_{0}}=z_{k}}z_{k}!\Big(\prod_{j_0=1}^{k_0}\frac{j_0^{r_{j_{0}}}}{r_{j_{0}}!}\Big)\prod_{j_0=1}^{k_0}\mathcal{B}_{r_{j_0}}(\mu_{j_0}te^{-j_0\Lambda})e^{\mu_{j_0}te^{-j_0\Lambda}},
\end{equation*}
where $\mathcal{B}_n(x)$ is the $n$th order Bell polynomial.
\end{corollary}
\begin{proposition}
The L\'evy measure of multivariate IGCP is given by
\begin{equation*}
\hat{\Pi}_{\bar{\mathcal{M}}}(A_1\times A_2\times \dots\times A_q)=\sum_{j_0=1}^{k_0}\mu_{j_0}\sum_{\bar{n}\succ\bar{0}}\sum_{\substack{\Omega(k_i,n_i)\\i=1,2,\dots,q}}\prod_{i=1}^{q}\Big(\prod_{j_i=1}^{k_i}\frac{(j_0\lambda_{ij_i})^{n_{ij_i}}}{n_{ij_i}!}e^{-j_0\lambda_{ij_i}}\Big)\mathbb{I}_{\{n_i\in A_i\}},
\end{equation*}
where $\Omega(k_i,n_i)=\{(n_{i1},n_{i2},\dots,n_{ik_i}):\sum_{j_i=1}^{k_i}j_in_{ij_i}=n_i,\,n_{ij_i}\in\mathbb{N}_0\}$.
\end{proposition}
\begin{proof}
Let $\Pi_{M_0}(\cdot)$ be the L\'evy measure of $\{M_0(t)\}_{t\ge0}$ and $p(\bar{n},t)$ be the pmf of MGCP. By using
Eq. (30.8) of Sato (1999), the L\'evy measure of multivariate IGCP can be obtained as follows:
{\small\begin{align*}
		\hat{\Pi}_{\bar{\mathcal{M}}}(A_1\times A_2\times\dots\times A_q)&=\int_{0}^{\infty}\sum_{\bar{n}\succ\bar{0}}p(\bar{n},s)\Big(\prod_{i=1}^{q}\mathbb{I}_{\{n_i\in A_i\}}\Big)\Pi_{M_0}(\mathrm{d}s)\\
		&=\int_{0}^{\infty}\sum_{\bar{n}\succ\bar{0}}\sum_{\substack{\Omega(k_i,n_i)\\i=1,2,\dots,q}}\Big(\prod_{i=1}^{q}\mathbb{I}_{\{n_i\in A_i\}}\prod_{j_i=1}^{k_i}\frac{(\lambda_{ij_i}s)^{n_{ij_i}}}{n_{ij_i!}}e^{-\lambda_{ij_i}s}\Big)\sum_{j_0=1}^{k_0}\mu_{j_0}\delta_{j_0}(\mathrm{d}s).
%		&=\sum_{j_0=1}^{k_0}\mu_{j_0}\sum_{\bar{n}\succ\bar{0}}\sum_{\substack{\Omega(k_i,n_i)\\i=1,2,\dots,q}}\prod_{i=1}^{q}\Big(\prod_{j_i=1}^{k_i}\frac{(j_0\lambda_{ij_i})^{n_{ij_i}}}{n_{ij_i}!}e^{-j_0\lambda_{ij_i}}\Big)\mathbb{I}_{\{n_i\in A_i\}}.
		\end{align*}}	
Finally, the result follows on integrating with respect to $\delta_{j_0}(\cdot)$, that is, $\int f(x)\delta_{j_0}(\mathrm{d}x)=f(j_0)$.
\end{proof}
\begin{proposition}
For $1\le i\le q$ and $1\le l\le q$, the covariance of $\{M_i(M_0(t))\}_{t\ge0}$ and $\{M_l(M_0(t))\}_{t\ge0}$ is given by
\begin{equation*}
\operatorname{Cov}(M_i(M_0(t)),M_l(M_0(t)))=\mathbb{I}_{\{i=l\}}\sum_{j_0=1}^{k_0}j_0\mu_{j_0}t\sum_{j_i=1}^{k_i}j_i^2\lambda_{ij_i}+\sum_{j_i=1}^{k_i}j_i\lambda_{ij_i}\sum_{j_l=1}^{k_l}j_l\lambda_{lj_l}\sum_{j_0=1}^{k_0}j_0^2\mu_{j_0}t.
\end{equation*}
\end{proposition}
\begin{proof}
For $i=l$, we have
\begin{align}
\operatorname{Cov}(M_i(M_0(t)),M_l(M_0(t)))&=\operatorname{Var}(M_i(M_0(t)))\nonumber\\
&=\sum_{j_i=1}^{k_i}j_i^2\lambda_{ij_i}\sum_{j_0=1}^{k_0}j_0\mu_{j_0}t+\Big(\sum_{j_i=1}^{k_i}j_i\lambda_{ij_i}\Big)^2\sum_{j_0=1}^{k_0}j_0^2\mu_{j_0}t\label{COV1}
\end{align}
which follows by using \eqref{IGCPVAR}.
For $i\ne l$, we have
\begin{align*}
\mathbb{E}(M_i(M_0(t))M_l(M_0(t)))&=\mathbb{E}(\mathbb{E}(M_i(M_0(t))M_l(M_0(t))|M_0(t)))\\
&=\mathbb{E}\Big(\sum_{j_i=1}^{k_i}j_i\lambda_{ij_i}\sum_{j_l=1}^{k_l}j_l\lambda_{lj_l}(M_0(t))^2\Big)\\
&=\sum_{j_i=1}^{k_i}j_i\lambda_{ij_i}\sum_{j_l=1}^{k_l}j_l\lambda_{lj_l}\Big(\sum_{j_0=1}^{k_0}j_0^2\mu_{j_0}t+\Big(\sum_{j_0=1}^{k_0}j_0\mu_{j_0}t\Big)^2\Big)
\end{align*}
 which follows from \eqref{covgcp}.
 So, by using \eqref{IGCPMEAN}, we have 
\begin{equation}\label{COV2}
 \operatorname{Cov}(M_i(M_0(t)),M_l(M_0(t)))=\sum_{j_i=1}^{k_i}j_i\lambda_{ij_i}\sum_{j_l=1}^{k_l}j_l\lambda_{lj_l}\sum_{j_0=1}^{k_0}j_0^2\mu_{j_0}t.
\end{equation}
Finally, the proof follows by using \eqref{COV1} and \eqref{COV2}.
\end{proof}

\begin{proposition}
	For any $1\le i\le q$ and $1\le l\le q$, the codifference of $\{M_i(M_0(t))\}_{t\ge0}$ and $\{M_l(M_0(t))\}_{t\ge0}$ is given by
	{\small\begin{align*}
			\tau(M_i(M_0(t)),M_l(M_0(t)))&=
			\sum_{j_0=1}^{k_0}\mu_{j_0}t\Big(2-\exp\Big(\sum_{j_i=1}^{k_i}j_0\lambda_{ij_i}(e^{\omega j_i}-1)\Big)+\exp\Big(\sum_{j_l=1}^{k_l}j_0\lambda_{lj_l}(e^{-\omega j_l}-1)\Big)\Big)\\
			&\ \ -\sum_{j_0=1}^{k_0}\mu_{j_0}t\Big(1-\exp\Big(\sum_{j_i=1}^{k_i}j_0\lambda_{ij_i}(e^{\omega j_i}-1)+\sum_{j_l=1}^{k_l}j_0\lambda_{lj_l}(e^{-\omega j_l}-1)\Big)\Big)\mathbb{I}_{\{i\ne l\}}.
	\end{align*}}
\end{proposition}
\begin{proof}
	For $i\ne l$, we have
{\footnotesize	\begin{align*}
		\mathbb{E}(e^{\omega(M_i(M_0(t))-M_l(M_0(t)))})&=\mathbb{E}\Big(\mathbb{E}\Big(e^{\omega(M_i(M_0(t))-M_l(M_0(t)))}|M_0(t)\Big)\Big)\\
		&=\mathbb{E}\bigg(\exp\bigg(\bigg(\sum_{j_i=1}^{k_i}\lambda_{ij_i}(e^{\omega j_i}-1)+\sum_{j_l=1}^{k_l}\lambda_{lj_l}(e^{-\omega j_l}-1)\bigg)M_0(t)\bigg)\bigg)\\
		&=\exp\Big(-\sum_{j_0=1}^{k_0}\mu_{j_0}t\Big(1-\exp\Big(\sum_{j_i=1}^{k_i}j_0\lambda_{ij_i}(e^{\omega j_i}-1)+\sum_{j_l=1}^{k_l}j_0\lambda_{lj_l}(e^{-\omega j_l}-1)\Big)\Big)\Big), 
	\end{align*}}
	which follows by using the \eqref{pgfmt}.
	For $i=l$, we get $\mathbb{E}(e^{\omega(M_i(M_0(t))-M_l(M_0(t)))})=1$. Finally, the result follows by using Eq. (1.7) of Kokoszka and Taqqu (1996). 
\end{proof}
\subsection{$q$-iterated GCP}
Let us consider $q$ independent GCPs $\{M_1(t)\}_{t\ge0}$, $\{M_2(t)\}_{t\ge0}$, $\dots$, $\{M_q(t)\}_{t\ge0}$ such that  $\{M_i(t)\}_{t\ge0}$ performs jumps of size $j_i$ with positive rate $\lambda_{j_i}$, $j_i=1,2,\dots,k_i$. The $q$-iterated GCP $\{\hat{M}^q(t)\}_{t\ge0}$ is defined as follows:
\begin{equation}\label{qigcprep}
\hat{M}^q(t)\coloneqq M(M_1(M_2(\dots(M_q(t))\dots))),\, t\ge0,
\end{equation}
where $\{M(t)\}_{t\ge0}$ is a GCP with positive rates $\lambda_j$, $j=1,2,\dots,k$ and it is independent of $\{M_i(t)\}_{t\ge0}$, $i=1,2,\dots,q$.
\begin{remark}
For $q=1$, the $q$-iterated GCP reduces to the IGCP. For $k=k_1=k_2=\dots=k_q=1$, the process defined in \eqref{qigcprep} reduces to the $q$-iterated Poisson process.
\end{remark}
Let $	\hat{p}^q(n,t)=\mathrm{Pr}\{\hat{M}^q(t)=n\}$, $n\ge0$ denote the state probabilities of $q$-iterated GCP. For $q=1$, we have
\begin{align*}
	\hat{p}^q(n,t)&=\mathrm{Pr}\{M(M_1(t))=n\}\\
	&=\sum_{s_1=0}^{\infty}\mathrm{Pr}\{M_1(t)=s_1\}\mathrm{Pr}\{M(s_1)=n\}\\
	&=\sum_{\Omega(k,n)}\Big(\prod_{j=1}^{k}\frac{\lambda_{j}^{n_{j}}}{n_{j}!}\Big)e^{-\mu t}\sum_{\sum_{j_1=1}^{k_1}r_{j_1}=z_{k}}z_{k}!\prod_{j_1=1}^{k_1}\frac{j_1^{r_{j_1}}}{r_{j_1}!}e^{\mu_{j_1}te^{-j_1\lambda}}\mathcal{B}_{r_{j_1}}(e^{-j_1\lambda}\mu_{j_1}t)
\end{align*}
which follows from \eqref{IGpmf}.

Now, for $q=2$, we have
\begin{align*}
	\hat{p}^q(n,t)&=\sum_{s_2=0}^{\infty}\mathrm{Pr}\{M_2(t)=s_2\}\mathrm{Pr}\{M(M_1(s_2))=n\}\\
	&=\sum_{s_2=0}^{\infty}\mathrm{Pr}\{M_2(t)=s_2\}\\
	&\hspace{1.7cm} \cdot\sum_{\Omega(k,n)}\Big(\prod_{j=1}^{k}\frac{\lambda_{j}^{n_{j}}}{n_{j}!}\Big)e^{-\mu s_2}\sum_{\sum_{j_1=1}^{k_1}r_{j_1}=z_{k}}z_{k}!\prod_{j_1=1}^{k_1}\frac{j_1^{r_{j_1}}}{r_{j_1}!}e^{\mu_{j_1}te^{-j_1\lambda}}\mathcal{B}_{r_{j_1}}(e^{-j_1\lambda}\mu_{j_1}s_2).
\end{align*}

Also, for $q=3$, we have
\begin{align*}
	\hat{p}^q(n,t)&=\sum_{s_3=0}^{\infty}\mathrm{Pr}\{M_3(t)=s_3\}\sum_{s_2=0}^{\infty}\mathrm{Pr}\{M_2(s_3)=s_2\}\mathrm{Pr}\{M(M_1(s_2))=n\}\\
	&=\sum_{s_2=0}^{\infty}\sum_{s_3=0}^{\infty}\mathrm{Pr}\{M_3(t)=s_3\}\mathrm{Pr}\{M_2(s_3)=s_2\}\mathrm{Pr}\{M(M_1(s_2))=n\}\\
	&=\sum_{s_2=0}^{\infty}\sum_{\Omega(k_2,n_2)}\Big(\prod_{j_2=1}^{k_2}\frac{\lambda_{j_2}^{n_{j_2}}}{n_{j_2}!}\Big)e^{-\lambda_3 t}
	\cdot\sum_{\sum_{j_3=1}^{k_3}r_{j_3}=z_{k_{2}}}z_{k_{2}}!\prod_{j_3=1}^{k_3}\frac{j_3^{r_{j_3}}}{r_{j_3}!}e^{\lambda_{j_3} te^{-j_3\lambda_2}}\mathcal{B}_{r_{j_3}}(e^{-j_3\lambda_2}\lambda_{j_3}t)\\
	&\hspace{1.8cm} \cdot\sum_{\Omega(k,n)}\Big(\prod_{j=1}^{k}\frac{\lambda_{j}^{n_{j}}}{n_{j}!}\Big)e^{-\mu s_2}\cdot\sum_{\sum_{j_1=1}^{k_1}r_{j_1}=z_{k}}z_{k}!\prod_{j_1=1}^{k_1}\frac{j_1^{r_{j_1}}}{r_{j_1}!}e^{\mu_{j_1}te^{-j_1\lambda}}\mathcal{B}_{r_{j_1}}(e^{-j_1\lambda}\mu_{j_1}s_2).
\end{align*}
Proceeding inductively, we can obtain the following result:
\begin{proposition}
The state of $q$-iterated GCP are given as follows:
 For $q$ even, we have
{\footnotesize\begin{align*}
	\hat{p}^q(n,t)&=\sum_{s_q=0}^{\infty}\mathrm{Pr}\{M_q(t)=s_q\}\Big(\prod_{i=1}^{\frac{q}{2}-1}\sum_{\Omega(k_{2i},s_{2i})}\Big(\prod_{j_{2i}=1}^{k_{2i}}\frac{\lambda_{j_{2i}}^{n_{j_{2i}}}}{n_{j_{2i}!}}\Big)\exp(-\lambda_{2i+1}s_{2(i+1)})\sum_{\sum_{j_{2i+1}=1}^{k_{2i+1}}r_{j_{2i+1}}=z_{k_{2i}}}z_{k_{2i}}!\\
	&\ \ \cdot\prod_{j_{2i+1}=1}^{k_{2i+1}}\frac{j_{2i+1}^{r_{j_{2i+1}}}}{r_{j_{2i+1}}!}\exp(\lambda_{j_{2i+1}}s_{2(i+1)}\exp(-j_{2i+1})\lambda_{2i})\mathcal{B}_{r_{j_{2i+1}}}(e^{-j_{2i+1}\lambda_{2i}}\lambda_{j_{2i+1}}s_{2(i+1)})\Big)\\
	&\ \ \cdot\sum_{\Omega(k,n)}\Big(\prod_{j=1}^{k}\frac{\lambda_j^{n_j}}{n_j!}\Big)e^{-s_2\lambda_1}\sum_{\sum_{j_1=1}^{k_1}r_{j_1}=z_k}z_k!\Big(\prod_{j_1=1}^{k_1}\frac{j_1^{r_{j_1}}}{r_{j_1}!}\Big)\exp(e^{-j_1\lambda}\lambda_{j_1}s_2)\mathcal{B}_{r_{j_1}}(e^{-j_1\lambda}\lambda_{j_1}s_2)
\end{align*}}
 and for $q$ odd, we have
{\footnotesize\begin{align*}
\hat{p}^q(n,t)&=\sum_{s_{q-1}=0}^{\infty}\sum_{\Omega(k_{q-1},s_{q-1})}\Big(\prod_{j_{q-1}=1}^{k_{q-1}}\frac{\lambda_{j_{q-1}}^{n_{j_{q-1}}}}{n_{j_{q-1}}!}\Big)e^{-t\lambda_q}\sum_{\sum_{j_q=1}^{k_q}r_{j_q}=z_{k_{q-1}}}z_{k_{q-1}}!\Big(\prod_{j_q=1}^{k_q}\frac{j_q^{r_{j_q}}}{r_{j_q}!}\exp(\lambda_{j_q}te^{-j_q\lambda_{q-1}})\\
&\ \ \cdot\mathcal{B}_{r_{j_q}}(\lambda_{j_q}te^{-j_q\lambda_{q-1}})\Big)\bigg(\prod_{i=1}^{\lfloor\frac{q}{2}\rfloor-1}\sum_{s_{2i}=0}^{\infty}\sum_{\Omega(k_{2i},s_{2i})}\Big(\prod_{j_{2i}=1}^{k_{2i}}\frac{\lambda_{j_{2i}}^{n_{j_{2i}}}}{n_{j_{2i}!}}\Big)\exp(-\lambda_{2i+1}s_{2(i+1)})\sum_{\sum_{j_{2i+1}=1}^{k_{2i+1}}r_{j_{2i+1}}=z_{k_{2i}}}z_{k_{2i}}!\\
&\ \ \cdot\prod_{j_{2i+1}=1}^{k_{2i+1}}\frac{j_{2i+1}^{r_{j_{2i+1}}}}{r_{j_{2i+1}}!}\exp(\lambda_{j_{2i+1}}s_{2(i+1)}\exp(-j_{2i+1})\lambda_{2i})\mathcal{B}_{r_{j_{2i+1}}}(e^{-j_{2i+1}\lambda_{2i}}\lambda_{j_{2i+1}}s_{2(i+1)})\bigg)\\
&\ \ \cdot\sum_{\Omega(k,n)}\Big(\prod_{j=1}^{k}\frac{\lambda_j^{n_j}}{n_j!}\Big)e^{-s_2\lambda_1}\sum_{\sum_{j_1=1}^{k_1}r_{j_1}=z_k}z_k!\Big(\prod_{j_1=1}^{k_1}\frac{j_1^{r_{j_1}}}{r_{j_1}!}\Big)\exp(e^{-j_1\lambda}\lambda_{j_1}s_2)\mathcal{B}_{r_{j_1}}(e^{-j_1\lambda}\lambda_{j_1}s_2).
\end{align*}}
\end{proposition}

\begin{proposition}
For $|u|\le1$, the pgf $\hat{G}^q(u,t)$ of $q$-iterated GCP is given by
{\small\begin{equation*}
\hat{G}^q(u,t)=\exp\Big(-\sum_{j_q=1}^{k_q}\lambda_{j_q}t\Big(1-\exp\Big(-j_q\sum_{j_{q-1}=1}^{k_{q-1}}\lambda_{j_{q-1}}\Big(\dots\Big(1-\exp\Big(-j_1\sum_{j=1}^{k}\lambda_j(1-u^j)\Big)\Big)\dots\Big)\Big)\Big)\Big).
\end{equation*}}
\end{proposition}
\begin{proof}
For $q=1$, by using \eqref{IGDCpgf},  we have
\begin{align*}
\mathbb{E}(u^{M(M_1(t))})=\exp\Big(-\sum_{j_1=1}^{k_1}\lambda_{j_1}t\Big(1-\exp\Big(-j_1\sum_{j=1}^{k}\lambda_j(1-u^j)\Big)\Big)\Big).
\end{align*}
Similarly, for $q=2$ and $q=3$, we get
{\small\begin{align*}
\mathbb{E}(u^{M(M_1(M_2(t)))})&=\mathbb{E}\Big(\exp\Big(-M_2(t)\sum_{j_1=1}^{k_1}\lambda_{j_1}\Big(1-\exp\Big(-j_1\sum_{j=1}^{k}\lambda_j(1-u^j)\Big)\Big)\Big)\Big)\\
&=\exp\Big(-\sum_{j_2=1}^{k_2}\lambda_{j_2}t\Big(1-\exp\Big(-j_2\sum_{j_1=1}^{k_1}\lambda_{j_1}\Big(1-\exp\Big(-j_1\sum_{j=1}^{k}\lambda_j(1-u^j)\Big)\Big)\Big)\Big)\Big)
\end{align*}}
and 
{\footnotesize\begin{align*}
\mathbb{E}&(u^{M(M_1(M_2(M_3(t))))})\\
&=\exp\Big(-M_3(t)\sum_{j_2=1}^{k_2}\lambda_{j_2}\Big(1-\exp\Big(-j_2\sum_{j_1=1}^{k_1}\lambda_{j_1}\Big(1-\exp\Big(-j_1\sum_{j=1}^{k}\lambda_j(1-u^j)\Big)\Big)\Big)\Big)\Big)\\
&=\exp\Big(-\sum_{j_3=1}^{k_3}\lambda_{j_3}t\Big(1 -\exp\Big(-j_3\sum_{j_2=1}^{k_2}\lambda_{j_2}\Big(1-\exp\Big(-j_2\sum_{j_1=1}^{k_1}\lambda_{j_1}\Big(1-\exp\Big(-j_1\sum_{j=1}^{k}\lambda_j(1-u^j)\Big)\Big)\Big)\Big)\Big)\Big)\Big),
\end{align*}}
respectively. Proceeding inductively, we obtain the required result.
\end{proof}
By using \eqref{covgcp} and Theorem 2.1 of Leonenko \textit{et al.} (2014), the mean and variance of $q$-iterated GCP can be obtained as follows:
\begin{align*}
\mathbb{E}(\hat{M}^q(t))&=\mathbb{E}(M(1))\mathbb{E}(M_1(M_2(\dots(M_q(t))\dots)))\\
&=\mathbb{E}(M(1))\mathbb{E}(M_1(1))\mathbb{E}(M_2(M_3(\dots(M_q(t))\dots)))\\
&\ \, \vdots\\
&=\mathbb{E}(M(1))\Big(\prod_{i=1}^{q-1}\mathbb{E}(M_i(1))\Big)\mathbb{E}(M_q(t))\\
&=\sum_{j=1}^{k}j\lambda_j\Big(\prod_{i=1}^{q}\sum_{j_i=1}^{k_i}j_i\lambda_{j_i}\Big)t
\end{align*}
and
{\small\begin{align*}
\operatorname{Var}(\hat{M}^q(t))&=\operatorname{Var}(M_1(1))\mathbb{E}(M_1(M_2(\dots(M_q(t))\dots))) +(\mathbb{E}(M(1)))^2\operatorname{Var}(M_1(M_2(\dots(M_q(t))\dots)))\\
&\ \, \vdots\\
&=\operatorname{Var}(M_1(1))\mathbb{E}(M_1(M_2(\dots(M_q(t))\dots)))+(\mathbb{E}(M(1)))^2\Big(\prod_{r=1}^{q-1}(\mathbb{E}(M_r(1)))^2\Big)\operatorname{Var}(M_q(t))\\
&\ \  +(\mathbb{E}(M(1)))^2\sum_{r=1}^{q}\Big(\operatorname{Var}(M_r(1))\Big(\prod_{l=1}^{r-1}(\mathbb{E}(M_l(1)))^2\Big)\mathbb{E}(M_{r+1}(M_{r+2}(\dots(M_q(t))\dots)))\Big)\\
&=\sum_{j=1}^{k}j^2\lambda_j t\Big(\prod_{i=1}^{q}\sum_{j_i=1}^{k_i}j_i\lambda_{j_i}\Big)+\Big(\sum_{j=1}^{k}j\lambda_j\Big)^2\Big(\prod_{r=1}^{q-1}\Big(\sum_{j_r=1}^{k_r}j_r\lambda_{j_r}\Big)^2\Big)\sum_{j_q=1}^{k_q}j_q^2\lambda_{j_q}t\\
&\ \ +\Big(\sum_{j=1}^{k}j\lambda_j\Big)^2\sum_{r=1}^{q-1}\Big(\sum_{j_r=1}^{k_r}j_r^2\lambda_{j_r}t\Big(\prod_{l=1}^{r-1}\Big(\sum_{j_l=1}^{k_l}j_l\lambda_{j_l}\Big)^2\Big)\Big(\prod_{i=r+1}^{q}j_i\lambda_{j_i}\Big)\Big),
\end{align*}}
respectively.
\begin{remark}
As $
\operatorname{Var}(\hat{M}^q(t))-\mathbb{E}(\hat{M}^q(t))>0$,
the $q$-iterated GCP is overdispersed.
\end{remark}

\section{A time-changed variant of IGCP}
In this section, we study a  time-changed variant of the IGCP by subordinating it with the first hitting time of a stable subordinator. We call it the time-changed IGCP. It is defined as follows:
\begin{equation}\label{IGFCPDEF}
\hat{M}^\alpha(t)\coloneqq \hat{M}((Y^\alpha(t))),\, t\ge0,
\end{equation}
where $\{Y^\alpha(t)\}_{t\ge0}$, $0<\alpha<1$ is an inverse $\alpha$-stable subordinator independent of the IGCP $\{\hat{M}(t)\}_{t\ge0}$. 
\begin{remark}
Note that $\hat{M}^\alpha(t)=M(M^\alpha_0(t))$, $t\ge0$, where $\{M_0^\alpha(t)\}_{t\ge0}$  is the GFCP. So, the time-changed IGCP is the GCP subordinated by an independent GFCP $\{M_0^\alpha(t)\}_{t\ge0}$.
\end{remark}
The pgf of time-changed IGCP can be obtained as follows:
\begin{align*}
\mathbb{E}(u^{\hat{M}^\alpha(t)})&=\mathbb{E}(\mathbb{E}(u^{\hat{M}^\alpha(t)}|M_0^\alpha(t)))\\
&=\mathbb{E}\Big(\exp\Big(-M_0^\alpha(t)\sum_{j=1}^{k}\lambda_j(1-u^j)\Big)\Big)\\
&=E_{\alpha,1}\Big(\sum_{j_0=1}^{k_0}\mu_{j_0}t^\alpha\Big(\exp\Big(j_0\sum_{j=1}^{k}\lambda_j(u^j-1)\Big)-1\Big)\Big),\ |u|\le1,
\end{align*}
where the last step follows by using Eq. (14) of Kataria and Khandakar (2022).
\begin{theorem}\label{thmIgfpmf}
The state probabilities $\hat{q}^\alpha(n,t)=\mathrm{Pr}\{\hat{M}^\alpha(t)=n\}$, $n\ge0$ of time-changed IGCP are given by
{\footnotesize
	\begin{equation*}
		\hat{q}^\alpha(n,t)=\sum_{\Omega(k,n)}\Big(\prod_{j=1}^{k}\frac{\lambda_{j}^{n_{j}}}{n_{j}!}\Big)\sum_{\sum_{j_0=1}^{k_0}r_{j_0}=z_{k}}z_{k}!\Big(\prod_{j_0=1}^{k_0}\frac{j_0^{r_{j_0}}}{r_{j_0}!}\Big)\sum_{\substack{x_{j_0}\ge0\\1\le j_0\le k_0}}\Big(\prod_{j_0=1}^{k_0}\frac{x_{j_0}^{r_{j_0}}(\mu_{j_0}t^\alpha e^{-j_0\lambda})^{x_{j_0}}}{x_{j_0}!}\Big)z_{k_0}! E_{\alpha,\alpha z_{k_0}+1}^{z_{k_0}+1}(-\mu t^\alpha),
\end{equation*}}
where  $\Omega(k,n)=\{(n_1,n_2,\dots,n_k):\sum_{j=1}^{k}jn_j=n,\, n_j\in\mathbb{N}_0\}$, $z_k=n_1+n_2+\dots+n_k$ and $z_{k_0}=x_1+x_2+\dots+x_{k_0}$.
\end{theorem}
\begin{proof}
By using \eqref{IGFCPDEF}, we have
\begin{align*}
\hat{q}^\alpha(n,t)&=\sum_{m=0}^{\infty}\mathrm{Pr}\{M(m)=n\}\mathrm{Pr}\{M_0^\alpha(t)=m\}\\
&=\sum_{\Omega(k,n)}\Big(\prod_{j=1}^{k}\frac{\lambda_j^{n_j}}{n_j!}\Big)\sum_{m=0}^{\infty}m^{z_{k}}e^{-m\lambda}\sum_{r=0}^{m}\sum_{\substack{x_1+x_2+\dots+x_{k_0}=r\\
		x_1+2x_2+\dots+k_0x_{k_0}=m}}r!\Big(\prod_{j_0=1}^{k_0}\frac{(\mu_{j_0}t^\alpha)^{x_{j_0}}}{x_{j_0}!}\Big)E_{\alpha,r\alpha+1}^{r+1}(-\mu t^\alpha)\\
&=\sum_{\Omega(k,n)}\Big(\prod_{j=1}^{k}\frac{\lambda_j^{n_j}}{n_j!}\Big)\sum_{\substack{x_{j_0}\ge0\\1\le j_0\le k_0}}\Big(\sum_{j_0=1}^{k_0}j_0x_{j_0}\Big)^{z_{k}}z_{k_0}!\Big(\prod_{j_0=1}^{k_0}\frac{(\mu_{j_0}t^\alpha e^{-j_0\lambda})^{x_{j_0}}}{x_{j_0}!}\Big)E_{\alpha,z_{k_0}\alpha+1}^{z_{k_0}+1}(-\mu t^\alpha),
\end{align*}
where second step follows by using \eqref{GFCPPMF} and \eqref{p(n,t)}. Finally, the proof follows by using multinomial theorem.
\end{proof}
For an alternate proof of Theorem \ref{thmIgfpmf}, we refer the reader to Appendix A4.
\begin{proposition}\label{prpdefIgfcp}
The state probabilities $\hat{q}^\alpha(n,t)$, $n\ge0$ solve the following system of differential equations:
\begin{equation*}
\frac{\mathrm{d}^\alpha}{\mathrm{d}t^\alpha}\hat{q}^\alpha(n,t)=-\sum_{j_0=1}^{k_0}\mu_{j_0}\hat{q}^\alpha(n,t)+\sum_{j_0=1}^{k_0}\mu_{j_0}e^{-j_0\lambda}\sum_{m=0}^{n}\sum_{\Omega(k,m)}\Big(\prod_{j=1}^{k}\frac{(j_0\lambda_j)^{x_j}}{x_j!}\Big)\hat{q}^\alpha(n-m,t)
\end{equation*}
with $\hat{q}^\alpha(n,0)=\delta_n(0)$.
Here, $\Omega(k,m)=\{(x_1,x_2,\dots,x_{k}):\sum_{j=1}^{k}jx_j=m,\, x_j\in\mathbb{N}_0\}$.
\end{proposition}
\begin{proof}
 From \eqref{IGFCPDEF}, we have
\begin{equation}\label{IGFCPDEFf}
\hat{q}^\alpha(n,t)=\int_{0}^{\infty}\hat{p}(n,x)h(x,t)\,\mathrm{d}x,
\end{equation}
where $h(x,t)$ is the density of inverse $\alpha$-stable subordinator and $\hat{p}(n,x)$ is the pmf of IGCP. 
On taking the Laplace transform of \eqref{IGFCPDEFf}, we get
\begin{equation}\label{tt}
\tilde{\hat{q}}^\alpha(n,s)=s^{\alpha-1}\int_{0}^{\infty}\hat{p}(n,x)e^{-s^\alpha x}\mathrm{d}x=s^{\alpha-1}\tilde{\hat{p}}(n,s^\alpha)
\end{equation}
which follows by using Eq. (3.13) of Meerschaert and Scheffler (2008).
Also, on taking the Laplace transform with respect to $t$ on both sides of  \eqref{IGCPDE}, we get
\begin{equation*}
s\tilde{\hat{p}}(n,s)-\hat{p}(n,0)=-\sum_{j_0=1}^{k_0}\mu_{j_0}\tilde{\hat{p}}(n,s)+\sum_{j_0=1}^{k_0}\mu_{j_0}e^{-j_0\lambda}\sum_{m=0}^{n}\sum_{\Omega(k,m)}\Big(\prod_{j=1}^{k}\frac{(j_0\lambda_j)^{x_j}}{x_j!}\Big)\tilde{\hat{p}}(n-m,s)
\end{equation*}
which by using \eqref{tt} becomes
{\small\begin{equation*}
s^\alpha\tilde{\hat{q}}^\alpha(n,s)-s^{\alpha-1}\hat{q}^\alpha(n,0)=-\sum_{j_0=1}^{k_0}\mu_{j_0}\tilde{\hat{q}}^\alpha(n,s)+\sum_{j_0=1}^{k_0}\mu_{j_0}e^{-j_0\lambda}\sum_{m=0}^{n}\sum_{\Omega(k,m)}\Big(\prod_{j=1}^{k}\frac{(j_0\lambda_j)^{x_j}}{x_j!}\Big)\tilde{\hat{q}}^\alpha(n-m,s).
\end{equation*}}
On taking the inverse Laplace transform on both sides of the above equation, we get the required result.
\end{proof}
 For an alternate proof of the Proposition \ref{prpdefIgfcp}, we refer the reader to Appendix A5.

By using Theorem 2.1 of Leonenko \textit{et al.} (2014), the mean and variance of time-changed IGCP can be obtained as follows:
\begin{equation*}
\mathbb{E}(\hat{M}^\alpha(t))=\mathbb{E}(\hat{M}(1))\mathbb{E}(Y^\alpha(t))=\frac{S t^\alpha}{\Gamma(\alpha+1)}
\end{equation*}
and
\begin{equation*}
\operatorname{Var}(\hat{M}^\alpha(t))=(\mathbb{E}(\hat{M}(1)))^2\operatorname{Var}(Y^\alpha(t))+\operatorname{Var}(\hat{M}(1))\mathbb{E}(Y^\alpha(t))=\frac{Rt^{2\alpha}+T t^\alpha}{\Gamma(\alpha+1)},
\end{equation*}
where we have used \eqref{meanvarinv}, \eqref{IGCPMEAN} and \eqref{IGCPVAR}.
Here, $R=S^2\Big(\frac{2}{\Gamma(2\alpha+1)}-\frac{1}{\Gamma^2(\alpha+1)}\Big)$.
\begin{remark}
Alternatively, by using \eqref{meanvarinv}, \eqref{meangfcp} and \eqref{vargfcp}, the mean and variance of time-changed IGCP can be obtained as follows:
$	\mathbb{E}(\hat{M}^\alpha(t))=\mathbb{E}(M(1))\mathbb{E}(M_0^\alpha(t))$ and $	\operatorname{Var}(\hat{M}^\alpha(t))=(\mathbb{E}(M(1)))^2\operatorname{Var}(M_0^\alpha(t))+\operatorname{Var}(M(1))\mathbb{E}(M_0^\alpha(t))$.
\end{remark}
 For $0<s<t$, let the correlation function for a non-stationary stochastic process satisfies
 \begin{equation*}
 \operatorname{Corr}(X(s),X(t))\sim c(s)t^{-\theta}, \ \ \text{as $t\to \infty$}
 \end{equation*}
for some $c(s)>0$. Then, the process $\{X(t)\}_{t\ge0}$ is said to exhibit the LRD property if $\theta\in(0,1)$ and the SRD property if $\theta\in(1,2)$. 

For fixed $s$ and large $t$, the covariance of $\{\hat{M}^\alpha(t)\}_{t\ge0}$
can be obtained as follows:
 \begin{equation*}
	\operatorname{Cov}(\hat{M}^\alpha(t), \hat{M}^\alpha(s))=\operatorname{Var}(\hat{M}(1))\mathbb{E}(Y^\alpha(s))+(\mathbb{E}(\hat{M}(1)))^2 \operatorname{Cov}(Y^\alpha(s),Y^\alpha(t))
\end{equation*}
which follows by using Theorem 2.1 of Leonenko \textit{et al.} (2014).
Now, by using \eqref{meanvarinv}, \eqref{covinv}, \eqref{IGCPMEAN} and \eqref{IGCPVAR}, we get
\begin{equation*}
\operatorname{Cov}(\hat{M}^\alpha(s),\hat{M}^\alpha(t))\sim \frac{Ts^\alpha}{
\Gamma(\alpha+1)}+\frac{S^2}{\Gamma^2(\alpha+1)}\Big(\alpha s^{2\alpha}B(\alpha,\alpha+1)-\frac{\alpha^2}{(\alpha+1)}\frac{s^{\alpha+1}}{t^{1-\alpha}}\Big),\ \ \text{as $t\to\infty$}.
\end{equation*}
	\begin{theorem}
The time-changed IGCP exhibits the LRD property.
\end{theorem}
\begin{proof}
For fixed $s$ and large $t$, we have
\begin{align*}
\operatorname{Corr}(\hat{M}^\alpha(s),\hat{M}^\alpha(t))&\sim\frac{Ts^\alpha+S^2\Big(\alpha s^{2\alpha}B(\alpha,\alpha+1)-\frac{\alpha^2}{(\alpha+1)}\frac{s^{\alpha+1}}{t^{1-\alpha}}\Big)}{\sqrt{\operatorname{Var}(\hat{M}^\alpha(s))}\sqrt{Rt^{2\alpha}+T t^\alpha}}\\
&\sim c(s)t^{-\alpha},
\end{align*}
where
\begin{equation*}
c(s)=\frac{Ts^\alpha\Gamma(\alpha+1)\Gamma(2\alpha+1)+S^2s^{2\alpha}\Gamma^2(\alpha+1)}{\Gamma^{3/2}(\alpha+1)\Gamma(2\alpha+1)\sqrt{R\operatorname{Var}(\hat{M}^\alpha(s))}}.
\end{equation*} 
As $\alpha\in(0,1)$, the time-changed IGCP exhibits the LRD property.
\end{proof}
For fixed $h\ge0$, the increment process of time-changed IGCP is defined as 
\begin{equation*}
	\hat{Z}^\alpha_h(t)\coloneqq\hat{M}^\alpha(t+h)-\hat{M}^\alpha(t),\, t\ge0.
\end{equation*}
\begin{proposition}
	The increment process $\{\hat{Z}^\alpha_h(t)\}_{t\ge0}$ exhibits the SRD property.
\end{proposition}
\begin{proof}
	The proof follows similar lines to that of Theorem 2 of Kataria and Khandakar (2022). Hence, it is omitted.
	%Let $s>0$ be fix such that $0<s+h\le t$ then
	%\begin{align*}
	%\operatorname{Cov}(\hat{Z}^\alpha(h,s),\hat{Z}^\alpha(h,t))&=\operatorname{Cov}(\hat{M}^\alpha(s+h),\hat{M}^\alpha(t+h))+\operatorname{Cov}(\hat{M}^\alpha(s),\hat{M}^\alpha(t))\\
	%&\sim (s+h)^\alpha T+S^2\Big(\alpha (s+h)^{2\alpha}B(\alpha,\alpha+1)-\frac{\alpha}{(\alpha+1)}\frac{(s+h)^{\alpha+1}}{(t+h)^{1-\alpha}}\Big),\ \ \text{as $t\to\infty$}.
	%\end{align*}
	%Thus, we have
	%\begin{align*}
	%\operatorname{Corr}(\hat{Z}^\alpha(h,s),\hat{Z}^\alpha(h,t))&=\frac{\operatorname{Cov}(\hat{Z}^\alpha(h,s),\hat{Z}^\alpha(h,t))}{\sqrt{\operatorname{Var}(\hat{Z}^\alpha(h,s))}\sqrt{\operatorname{Var}(\hat{Z}^\alpha(h,t))}}\\
	%&\sim c(s)t^{-\alpha},
	%\end{align*}
	%where 
	%\begin{equation*}
	%c(s)=\frac{num}{den}
	%\end{equation*}
\end{proof}
By using \eqref{IGCPMEAN} and the strong law of large numbers, we get the next result.
\begin{lemma}\label{lemas}
The following asymptotic result holds true for IGCP:
\begin{equation*}
\lim_{t\to\infty}\frac{\hat{M}(t)}{t}=\sum_{j=1}^{k}j\lambda_j\sum_{j_0=1}^{k_0}j_0\mu_{j_0}, \ \ \text{in probability}.
\end{equation*}
\end{lemma}
\begin{proposition}
The one-dimensional distributions of time-changed IGCP are not infinitely divisible.
\end{proposition}
\begin{proof}
By using the self-similarity property of inverse $\alpha$-stable subordinator, we have
\begin{equation*}
\hat{M}^\alpha(t)\overset{d}{=}\hat{M}(t^\alpha Y^\alpha(1)).
\end{equation*}
Thus, 
\begin{align*}
\lim_{t\to\infty}\frac{\hat{M}^\alpha(t)}{t}&\overset{d}{=}\lim_{t\to\infty}\frac{\hat{M}(t^\alpha Y^\alpha(1))}{t^\alpha}\\
&=Y^\alpha(1)\lim_{t\to\infty}\frac{\hat{M}(t^\alpha Y^\alpha(1))}{t^\alpha Y^\alpha(1)}\\
&\overset{d}{=}Y^\alpha(1)\sum_{j=1}^{k}j\lambda_j\sum_{j_0=1}^{k_0}j_0\mu_{j_0},
\end{align*}
where the last step follows by using Lemma \ref{lemas}. Let us assume that $\{\hat{M}^\alpha(t)\}_{t\ge0}$ is infinitely divisible. Then, $\hat{M}^\alpha(t)/t^\alpha$ is also infinitely divisible. Consequently, $Y^\alpha(1)$ is infinitely divisible since $\lim_{t\to\infty}\hat{M}^\alpha(t)/t^\alpha$ is infinitely divisible which follows from a result of Steutel and van Harn (2004), p. 94. This leads to a contradiction as
$Y^\alpha(1)$ is not infinitely divisible (see Vellaisamy and Kumar (2018)).
\end{proof}
\begin{proposition}
Let $r\ge1$. The $r$th factorial moment $\Phi^\alpha(r,t)=\mathbb{E}(\hat{M}^\alpha(t)(\hat{M}^\alpha(t)-1)\dots(\hat{M}^\alpha(t)-r+1))$ of time-changed IGCP is given by
\begin{equation*}
\Phi^\alpha(r,t)=\sum_{n=1}^{r}\frac{r!}{\Gamma(n\alpha+1)}t^{n\alpha}\sum_{\substack{\sum_{l=1}^{n}m_l=r\\m_l\in\mathbb{N}}}\prod_{l=1}^{n}\sum_{j_0=1}^{k_0}\mu_{j_0}\sum_{s=0}^{\infty}\frac{j_0^s}{s!}\sum_{\substack{\sum_{i=1}^{s}x_i=m_l\\x_i\in\mathbb{N}_0}}\prod_{i=1}^{s}\frac{1}{x_i!}\sum_{j=1}^{k}\lambda_j(j)_{x_i},
\end{equation*}
where $(j)_{x_i}=j(j-1)\dots(j-x_i+1)$ denotes the falling factorial.
\end{proposition}
\begin{proof}
We have
$
\Phi^\alpha(r,t)=\frac{\partial^r}{\partial u^r}\mathbb{E}(u^{\hat{M}^\alpha(t)})|_{u=1}.
$
By using the $r$th derivative of the composition of two functions (see Johnson (2002),
Eq. (3.3)), we get
\begin{align}
\Phi^\alpha(r,t)&=\sum_{n=0}^{r}\frac{1}{n!}E_{\alpha,1}^{(n)}\Big(\sum_{j_0=1}^{k_0}\mu_{j_0}t^\alpha\Big(\exp\Big(j_0\sum_{j=1}^{k}\lambda_j(u^j-1)\Big)-1\Big)\Big)\nonumber\\
&\hspace{4cm} \cdot B_{r,n}\Big(\sum_{j_0=1}^{k_0}\mu_{j_0}t^\alpha\Big(\exp\Big(j_0\sum_{j=1}^{k}\lambda_j(u^j-1)\Big)-1\Big)\Big)\Big|_{u=1},\label{nnnn}
\end{align}
where 
\begin{align}\label{mmmm}
E_{\alpha,1}^{(n)}\Big(\sum_{j_0=1}^{k_0}\mu_{j_0}t^\alpha\Big(\exp&\Big(j_0\sum_{j=1}^{k}\lambda_j(u^j-1)\Big)-1\Big)\Big)\Big|_{u=1}\nonumber\\
&=n!E_{\alpha,n\alpha+1}^{n+1}\Big(\sum_{j_0=1}^{k_0}\mu_{j_0}t^\alpha\Big(\exp\Big(j_0\sum_{j=1}^{k}\lambda_j(u^j-1)\Big)-1\Big)\Big)\Big|_{u=1}\nonumber\\
&=\frac{n!}{\Gamma(n\alpha+1)}
\end{align}
and 
\begin{align*}
B_{r,n}\Big(&\sum_{j_0=1}^{k_0}\mu_{j_0}t^\alpha\Big(\exp\Big(j_0\sum_{j=1}^{k}\lambda_j(u^j-1)\Big)-1\Big)\Big)\Big|_{u=1}\\
&=\sum_{m=0}^{n}\binom{n}{m}\Big(-\sum_{j_0=1}^{k_0}\mu_{j_0}t^\alpha\Big(\exp\Big(j_0\sum_{j=1}^{k}\lambda_j(u^j-1)\Big)-1\Big)\Big)^{n-m}\\
&\hspace{5.7cm} \cdot\frac{\mathrm{d}^r}{\mathrm{d}t^r}\Big(\sum_{j_0=1}^{k_0}\mu_{j_0}t^\alpha\Big(\exp\Big(j_0\sum_{j=1}^{k}\lambda_j(u^j-1)\Big)-1\Big)\Big)^{m}\Big|_{u=1}\\
&=\frac{\mathrm{d}^r}{\mathrm{d}t^r}\Big(\sum_{j_0=1}^{k_0}\mu_{j_0}t^\alpha\Big(\exp\Big(j_0\sum_{j=1}^{k}\lambda_j(u^j-1)\Big)-1\Big)\Big)^{m}\Big|_{u=1}.
\end{align*}
Now, by using Eq. (3.6) of Johnson (2002),  we get
\begin{align*}
\frac{\mathrm{d}^r}{\mathrm{d}t^r}&\Big(\sum_{j_0=1}^{k_0}\mu_{j_0}t^\alpha\Big(\exp\Big(j_0\sum_{j=1}^{k}\lambda_j(u^j-1)\Big)-1\Big)\Big)^{m}\Big|_{u=1}\\
&=r!t^{n\alpha}\sum_{\substack{\sum_{l=1}^{n}m_l=r\\m_l\in\mathbb{N}_0}}\prod_{l=1}^{n}\frac{1}{m_l!}\frac{\mathrm{d}^{m_l}}{\mathrm{d}u^{m_l}}\Big(\sum_{j_0=1}^{k_0}\mu_{j_0}\exp\Big(j_0\sum_{j=1}^{k}\lambda_j(u^j-1)\Big)\Big)\Big|_{u=1}\\
&=r!t^{n\alpha}\sum_{\substack{\sum_{l=1}^{n}m_l=r\\m_l\in\mathbb{N}_0}}\prod_{l=1}^{n}\frac{1}{m_l!}\sum_{j_0=1}^{k_0}\mu_{j_0}\sum_{s=0}^{\infty}\frac{j_0^s}{s!}\frac{\mathrm{d}^{m_l}}{\mathrm{d}u^{m_l}}\Big(\sum_{j=1}^{k}\lambda_j(u^j-1)\Big)^s\Big|_{u=1}\\
&=r!t^{n\alpha}\sum_{\substack{\sum_{l=1}^{n}m_l=r\\m_l\in\mathbb{N}_0}}\prod_{l=1}^{n}\frac{1}{m_l!}\sum_{j_0=1}^{k_0}\mu_{j_0}\sum_{s=0}^{\infty}\frac{j_0^s}{s!}m_l!\sum_{\substack{\sum_{i=1}^{s}x_i=m_l\\x_i\in\mathbb{N}_0}}\prod_{i=1}^{s}\frac{1}{x_i!}\frac{\mathrm{d}^{x_i}}{\mathrm{d}u^{x_i}}\Big(\sum_{j=1}^{k}\lambda_j(u^j-1)\Big)\Big|_{u=1}\\
&=r!t^{n\alpha}\sum_{\substack{\sum_{l=1}^{n}m_l=r\\m_l\in\mathbb{N}_0}}\prod_{l=1}^{n}\frac{1}{m_l!}\sum_{j_0=1}^{k_0}\mu_{j_0}\sum_{s=0}^{\infty}\frac{j_0^s}{s!}m_l!\sum_{\substack{\sum_{i=1}^{s}x_i=m_l\\x_i\in\mathbb{N}_0}}\prod_{i=1}^{s}\frac{1}{x_i!}\sum_{j=1}^{k}\lambda_j(j)_{x_i}.
\end{align*}
Thus, 
\begin{align}
B_{r,n}\Big(&\sum_{j_0=1}^{k_0}\mu_{j_0}t^\alpha\Big(\exp\Big(j_0\sum_{j=1}^{k}\lambda_j(u^j-1)\Big)-1\Big)\Big)\Big|_{u=1}\nonumber\\
&=r!t^{n\alpha}\sum_{\substack{\sum_{l=1}^{n}m_l=r\\m_l\in\mathbb{N}_0}}\prod_{l=1}^{n}\frac{1}{m_l!}\sum_{j_0=1}^{k_0}\mu_{j_0}\sum_{s=0}^{\infty}\frac{j_0^s}{s!}m_l!\sum_{\substack{\sum_{i=1}^{s}x_i=m_l\\x_i\in\mathbb{N}_0}}\prod_{i=1}^{s}\frac{1}{x_i!}\sum_{j=1}^{k}\lambda_j(j)_{x_i}.\label{mnmn}
\end{align}
Finally, the result follows on substituting \eqref{mmmm} and \eqref{mnmn} in \eqref{nnnn}. 
\end{proof}

\section*{Appendix}
\paragraph{A1}{\textit{Proof of Proposition \ref{prpde}}:}
From \eqref{Ipgfde}, we have
\begin{align}
	\frac{\partial}{\partial t}\hat{G}(u,t)&=\Big(-\mu+\sum_{j_0=1}^{k_0}\mu_{j_0}e^{-j_0\lambda}\sum_{r=0}^{\infty}j_0^r\sum_{\sum_{j=1}^{k}x_j=r}\prod_{j=1}^{k}\frac{(\lambda_{j}u^{j})^{x_{j}}}{x_{j}!}\Big)\hat{G}(u,t)\nonumber\\
	&=-\mu\sum_{n\ge0}u^n\hat{p}(n,t)+\sum_{j_0=1}^{k_0}\mu_{j_0}e^{-j_0\lambda}\sum_{\substack{x_{j}\ge0\\1\le j\le k}}\Big(\prod_{j=1}^{k}\frac{(j_0\lambda_{j}u^{j})^{x_{j}}}{x_{j}!}\Big)\sum_{n\ge0}u^n\hat{p}(n,t)\nonumber\\
	&=\sum_{n\ge0}u^n\Big(-\mu\hat{p}(n,t)+\sum_{j_0=1}^{k_0}\mu_{j_0}e^{-j_0\lambda}\sum_{m=0}^{n}\sum_{\Omega(k,m)}\Big(\prod_{j=1}^{k}\frac{(j_0\lambda_{j})^{x_{j}}}{x_{j}!}\Big)\hat{p}(n-m,t)\Big).\label{Am}
\end{align}
Also, we have
\begin{equation}\label{AA}
	\frac{\partial}{\partial t}\hat{G}(u,t)=\sum_{n\ge0}u^n\frac{\mathrm{d}}{\mathrm{d}t}\hat{p}(n,t).
\end{equation}
Finally, on comparing the coefficient of $u^n$ over the range $n\ge0$ on both sides of \eqref{Am} and \eqref{AA}, we get the required result.
\vspace{.5cm}
\paragraph{A2}{\textit{Proof of Proposition \ref{thmappen}}:}
From \eqref{MIGCPpgfde}, we have
\begin{align}
	\frac{\partial}{\partial t}G_{\bar{\mathcal{M}}}(\bar{u},t)&=\Big(-\mu +\sum_{j_0=1}^{k_0}\mu_{j_0}e^{-j_0\lambda}\sum_{r=0}^{\infty}\frac{(r\sum_{i=1}^{q}\sum_{j_i=1}^{k_i}\lambda_{ij_i}u_i^{j_i})^r}{r!}\Big)G_{\bar{\mathcal{M}}}(\bar{u},t)\nonumber\\
	&=\Big(-\mu +\sum_{j_0=1}^{k_0}\mu_{j_0}e^{-j_0\lambda}\sum_{r=0}^{\infty}\frac{j_0^r}{r!}\sum_{r_1+r_2+\dots+r_q=r}r!\prod_{i=1}^{q}\frac{(\sum_{j_i=1}^{k_i}\lambda_{ij_i}u_i^{j_i})^{r_i}}{r_i!}\Big)G_{\bar{\mathcal{M}}}(\bar{u},t)\nonumber\\
	%&=\Big(-\mu +\sum_{j_0=1}^{k_0}\mu_{j_0}e^{-j_0\lambda}\sum_{r=0}^{\infty}j_0^r\sum_{r_1+r_2+\dots+r_q=r}\prod_{i=1}^{q}\sum_{x_{i1}+x_{i2}+\dots+x_{ik_i}=r_i}\prod_{j_i=1}^{k_i}\frac{(\lambda_{ij_i}u_i^{j_i})^{x_{ij_i}}}{x_{ij_i}!}\Big)G_{\bar{\mathcal{M}}}(\bar{u},t)\nonumber\\
	&=\Big(-\mu +\sum_{j_0=1}^{k_0}\mu_{j_0}e^{-j_0\lambda}\sum_{\substack{r_i\ge0\\i=1,2,\dots,q}}\prod_{i=1}^{q}j_0^{r_i}\sum_{x_{i1}+x_{i2}+\dots+x_{ik_i}=r_i}\prod_{j_i=1}^{k_i}\frac{(\lambda_{ij_i}u_i^{j_i})^{x_{ij_i}}}{x_{ij_i}!}\Big)G_{\bar{\mathcal{M}}}(\bar{u},t)\nonumber\\ 
	&=\Big(-\mu +\sum_{j_0=1}^{k_0}\mu_{j_0}e^{-j_0\lambda}\sum_{\substack{m_i\ge0\\i=1,2,\dots,q}}\sum_{\substack{\Omega(k_i,m_i)\\i=1,2,\dots,q}}\prod_{i=1}^{q}u_i^{m_i}\prod_{j_i=1}^{k_i}\frac{(j_0\lambda_{ij_i})^{x_{ij_i}}}{x_{ij_i}!}\Big)G_{\bar{\mathcal{M}}}(\bar{u},t)\nonumber\\ 
	&=-\mu \sum_{\bar{n}\ge\bar{0}}\Big(\prod_{i=1}^{q}u_i^{n_i}\Big)p_{\bar{\mathcal{M}}}(\bar{n},t)+\sum_{j_0=1}^{k_0}\mu_{j_0}e^{-j_0\lambda}\nonumber\\
	&\hspace{2.5cm}\cdot\sum_{\substack{m_i\ge0\\i=1,2,\dots,q}}\sum_{\substack{\Omega(k_i,m_i)\\i=1,2,\dots,q}}\Big(\prod_{i=1}^{q}\prod_{j_i=1}^{k_i}\frac{(j_0\lambda_{ij_i})^{x_{ij_i}}}{x_{ij_i}!}\Big)\sum_{\bar{n}\ge\bar{0}}\Big(\prod_{i=1}^{q}u_i^{m_i+n_i}\Big)p_{\bar{\mathcal{M}}}(\bar{n},t)\nonumber\\ 
	&=-\mu\sum_{\bar{n}\ge\bar{0}}\Big(\prod_{i=1}^{q}u_i^{n_i}\Big) p_{\bar{\mathcal{M}}}(\bar{n},t)+\sum_{\bar{n}\ge\bar{0}}\Big(\prod_{i=1}^{q}u_i^{n_i}\Big)\sum_{j_0=1}^{k_0}\mu_{j_0}e^{-j_0\lambda}\nonumber\\
	&\hspace{5cm}\cdot \sum_{\substack{m_i\ge0\\i=1,2,\dots,q}}\sum_{\substack{\Omega(k_i,m_i)\\i=1,2,\dots,q}}\Big(\prod_{i=1}^{q}\prod_{j_i=1}^{k_i}\frac{(j_0\lambda_{ij_i})^{x_{ij_i}}}{x_{ij_i}!}\Big)p_{\bar{\mathcal{M}}}(\bar{n}-\bar{m},t). \label{A}
\end{align}
Also, we have 
\begin{equation}\label{AAbn}
	\frac{\partial}{\partial t}G_{\bar{\mathcal{M}}}(\bar{u},t)=\sum_{\bar{n}\ge\bar{0}}\Big(\prod_{i=1}^{q}u_i^{n_i}\Big)\frac{\mathrm{d}}{\mathrm{d}t} p_{\bar{\mathcal{M}}}(\bar{n},t).
\end{equation}
Finally, the result follows on comparing the coefficient of $u_1^{n_1}u_2^{n_2}\dots u_q^{n_q}$ over the range $\bar{n}\ge\bar{0}$ on both sides of \eqref{A} and \eqref{AAbn}.

\vspace{.5cm}
\paragraph{A3}{\textit{Proof of Proposition \ref{prpcgpart}}:}
By using \eqref{IGDCpgf}, we get
\begin{align*}
	\mathbb{E}(u^{D(t)})&=\exp\Big(-t\sum_{j_0=1}^{k_0}\mu_{j_0}\Big(1-\exp\Big(-j_0\sum_{j=1}^{k}\lambda_j(1-(\mathbb{E}(u^X))^j)\Big)\Big)\Big)\\
	&=\exp\Big(-t\sum_{j_0=1}^{k_0}\mu_{j_0}\Big(1-\exp\Big(-j_0\sum_{j=1}^{k}\lambda_j(1-\mathbb{E}(u^{X_1+X_2+\dots+X_j}))\Big)\Big)\Big)\\
	&=\exp\Big(-t\sum_{j_0=1}^{k_0}\mu_{j_0}\Big(1-\exp\Big(-j_0\sum_{j=1}^{k}\lambda_j\Big(1-\sum_{i=0}^{\infty}u^i\mathrm{Pr}\{X_1+\dots+X_j=i\}\Big)\Big)\Big)\Big)\\
	&=\exp\Big(-t\sum_{j_0=1}^{k_0}\mu_{j_0}\Big(1-\exp\Big(-j_0\sum_{j=1}^{k}\lambda_j\Big(1-\sum_{i=0}^{\infty}u^i\sum_{\substack{\sum_{m=1}^{j}r_m=i\\r_m\in\mathbb{N}_0}}\alpha_{r_1}\alpha_{r_2}\dots\alpha_{r_j}\Big)\Big)\Big)\Big)\\
	&=\exp\Big(-t\sum_{j_0=1}^{k_0}\mu_{j_0}\Big(1-\exp\Big(-j_0\sum_{j=1}^{k}\lambda_j\Big(1-\sum_{i=0}^{\infty}u^i\alpha^{*(j)}_i\Big)\Big)\Big)\Big)\\
	&=\exp\Big(-t\sum_{j_0=1}^{k_0}\mu_{j_0}\Big(1-\exp\Big(-j_0\sum_{j=1}^{k}\lambda_j\sum_{i=0}^{\infty}\alpha^{*(j)}_i(1-u^i)\Big)\Big)\Big),\, \text{(as $\sum_{i=0}^{\infty}\alpha^{*(j)}_i=1$)}\\
	&=\exp\Big(-t\sum_{j_0=1}^{k_0}\mu_{j_0}\Big(1-\exp\Big(-j_0\sum_{j=1}^{k}\lambda_j\sum_{i=1}^{\infty}\alpha^{*(j)}_i(1-u^i)\Big)\Big)\Big).
\end{align*}
This completes the proof.
\vspace{.5cm}
\paragraph{A4}{\textit{Proof of Theorem \ref{thmIgfpmf}}:}
Let $h(x,t)$ be the density of inverse $\alpha$-stable subordinator and $\hat{p}(n,x)$ be the pmf of IGCP. Then, from \eqref{IGFCPDEF}, we have
{\scriptsize\begin{align}
	\hat{q}^\alpha(n,t)&=\int_{0}^{\infty}\hat{p}(n,s)h(s,t)\,\mathrm{d}s\nonumber\\
	&=\sum_{\Omega(k,n)}\Big(\prod_{j=1}^{k}\frac{\lambda_{j}^{n_{j}}}{n_{j}!}\Big)\sum_{\sum_{j_0=1}^{k_0}r_{j_0}=z_{k}}z_{k}!\Big(\prod_{j_0=1}^{k_0}\frac{j_0^{r_{j_0}}}{r_{j_0}!}\Big)\sum_{\substack{x_{j_0}\ge0\\1\le j_0\le k_0}}\Big(\prod_{j_0=1}^{k_0}\frac{x_{j_0}^{r_{j_0}}(\mu_{j_0}e^{-j_0\lambda})^{x_{j_0}}}{x_{j_0}!}\Big)\int_{0}^{\infty}s^{z_{k_0}}e^{-\mu s}h(s,t)\mathrm{d}s\label{MKLN}
\end{align}}
On taking Laplace transform on both sides of \eqref{MKLN}, we get
{\scriptsize\begin{align*}
\tilde{\hat{q}}(n,w)&=\sum_{\Omega(k,n)}\Big(\prod_{j=1}^{k}\frac{\lambda_{j}^{n_{j}}}{n_{j}!}\Big)\sum_{\sum_{j_0=1}^{k_0}r_{j_0}=z_{k}}z_{k}!\Big(\prod_{j_0=1}^{k_0}\frac{j_0^{r_{j_0}}}{r_{j_0}!}\Big)\sum_{\substack{x_{j_0}\ge0\\1\le j_0\le k_0}}\Big(\prod_{j_0=1}^{k_0}\frac{x_{j_0}^{r_{j_0}}(\mu_{j_0}e^{-j_0\lambda})^{x_{j_0}}}{x_{j_0}!}\Big)\int_{0}^{\infty}s^{z_{k_0}}e^{-\mu s}w^{\alpha-1}e^{-sw^\alpha}\mathrm{d}s\\
&=\sum_{\Omega(k,n)}\Big(\prod_{j=1}^{k}\frac{\lambda_{j}^{n_{j}}}{n_{j}!}\Big)\sum_{\sum_{j_0=1}^{k_0}r_{j_0}=z_{k}}z_{k}!\Big(\prod_{j_0=1}^{k_0}\frac{j_0^{r_{j_0}}}{r_{j_0}!}\Big)\sum_{\substack{x_{j_0}\ge0\\1\le j_0\le k_0}}\Big(\prod_{j_0=1}^{k_0}\frac{x_{j_0}^{r_{j_0}}(\mu_{j_0}e^{-j_0\lambda})^{x_{j_0}}}{x_{j_0}!}\Big)\frac{w^{\alpha-1}\Gamma(z_{k_0}+1)}{(\mu+w^\alpha)^{z_{k_0}+1}}.
\end{align*}}
Now, on taking the inverse Laplace transform on both sides of the above equation, we get
{\scriptsize
\begin{equation*}
	\hat{q}^\alpha(n,t)=\sum_{\Omega(k,n)}\Big(\prod_{j=1}^{k}\frac{\lambda_{j}^{n_{j}}}{n_{j}!}\Big)\sum_{\sum_{j_0=1}^{k_0}r_{j_0}=z_{k}}z_{k}!\Big(\prod_{j_0=1}^{k_0}\frac{j_0^{r_{j_0}}}{r_{j_0}!}\Big)\sum_{\substack{x_{j_0}\ge0\\1\le j_0\le k_0}}\Big(\prod_{j_0=1}^{k_0}\frac{x_{j_0}^{r_{j_0}}(\mu_{j_0}t^\alpha e^{-j_0\lambda})^{x_{j_0}}}{x_{j_0}!}\Big)z_{k_0}!E_{\alpha,\alpha z_{k_0}+1}^{z_{k_0}+1}(-\mu t^\alpha),
\end{equation*}}
where in the last step, we have used \eqref{mi}.
This completes the proof.

\vspace{.5cm}
\paragraph{A5}{\textit{Proof of Proposition \ref{prpdefIgfcp}}:}
By using \eqref{IGFCPDEF}, we have
\begin{equation*}
	\hat{q}^\alpha(n,t)=\sum_{m=0}^{\infty}\mathrm{Pr}\{M(m)=n\}p^\alpha(m,t).
\end{equation*}
So, by using \eqref{gfcpdeq}, we get
\begin{align*}
\frac{\mathrm{d}^\alpha}{\mathrm{d}t^\alpha}\hat{q}^\alpha(n,t)&=\sum_{m=0}^{\infty}\mathrm{Pr}\{M(m)=n\}\Big(-\sum_{j_0=1}^{k_0}\mu_{j_0}(p^\alpha(m,t)-p^\alpha(m-j_0,t))\Big)\\
&=-\sum_{j_0=1}^{k_0}\mu_{j_0}\hat{q}^\alpha(n,t)+\sum_{j_0=1}^{k_0}\mu_{j_0}\mathrm{Pr}\{M(m)=n\}p^\alpha(m-j_0,t)\\
&=-\sum_{j_0=1}^{k_0}\mu_{j_0}\hat{q}^\alpha(n,t)+\sum_{j_0=1}^{k_0}\mu_{j_0}\sum_{r=0}^{\infty}\mathrm{Pr}\{M(j_0)=r\}\hat{q}^\alpha(n-r,t)
\end{align*}
which reduces to the required result on using \eqref{p(n,t)}.

\end{document}